\documentclass[11pt]{ip-journal}

\newtheorem{prop}{Proposition}[section]
\newtheorem{theo}[prop]{Theorem}
\newtheorem*{theo*}{Theorem}
\newtheorem{lemm}[prop]{Lemma}
\newtheorem{coro}[prop]{Corollary}

\newtheorem{defi}[prop]{Definition}
\newtheorem{rema}[prop]{Remark}

\theoremstyle{definition}

\newtheorem{conj}[prop]{Conjecture}

\newcommand{\RR}{\mathbf{R}}

\newcommand{\ZZ}{\mathbf{Z}}

\newcommand{\cA}{\mathcal A}

\newcommand{\cC}{\mathcal C}

\newcommand{\cH}{\mathcal H}

\newcommand{\cP}{\mathcal P}

\DeclareMathOperator{\tr}{tr}

\DeclareMathOperator{\graph}{graph}

\DeclareMathOperator{\spt}{spt}

\DeclareMathOperator{\dist}{dist}

\DeclareMathOperator{\inj}{inj}

\DeclareMathOperator{\Ric}{Ric}

\DeclareMathOperator{\secondfund}{II}

\newcommand{\pa}[2]{\frac{\partial #1}{\partial #2}}
\newcommand{\paop}[1]{\pa{}{#1}}
\newcommand{\td}[2]{\frac{d #1}{d #2}}
\newcommand{\bangle}[1]{\left\langle #1 \right\rangle}

\newcommand{\ep}{\varepsilon}

\numberwithin{equation}{section}

\begin{document}

\title{The dihedral rigidity conjecture for $n$-prisms}

\author{Chao Li}
\address{Courant institute of mathematical sciences, New York University, \\ 251 Mercer St, New York, NY 10012}
\email{chaoli@nyu.edu}

\begin{abstract}
    We prove the following comparison theorem for metrics with nonnegative scalar curvature, also known as the dihedral rigidity conjecture by Gromov: for $n\le 7$, if an $n$-dimensional prism has nonnegative scalar curvature and weakly mean convex faces, then its dihedral angle cannot be everywhere not larger than its Euclidean model, unless it is isometric to an Euclidean prism. The proof relies on constructing certain free boundary minimal hypersurface in a Riemannian polyhedron, and extending a dimension descent idea of Schoen-Yau. Our result is a localization of the positive mass theorem.
\end{abstract}


\maketitle

\section{Introduction}
In a paper \cite{Gromov2014Dirac} from 2014, Gromov proposed the first steps towards understanding Riemannian manifolds with scalar curvature bounded below. He suggested that a polyhedron comparison theorem should play a role analogous to that of the Alexandrov's triangle comparisons for spaces with sectional curvature lower bounds \cite{AleksandrovRerestovskiiNikolaev86}. Precisely, let $M^n$ be a convex polyhedron in Euclidean space, and $g$ a metric on $M$. Denote the Euclidean metric $g_0$. Gromov made the following conjecture (see section 2.2 of \cite{Gromov2014Dirac}, and section 7, Question $F_1$ of \cite{Gromov2018Adozen}):

\begin{conj}[The dihedral rigidity conjecture]\label{conj:dihedral.rigidity}
    Suppose $(M,g)$ has nonnegative scalar curvature and weakly mean convex faces, and along the intersection of any two adjacent faces, the dihedral angle of $(M,g)$ is not larger than the (constant) dihedral angle of $(M,g_0)$. Then $(M,g)$ is isometric to a flat Euclidean polyhedron.
\end{conj}

When $n=2$, Conjecture \ref{conj:dihedral.rigidity} follows directly from the Gauss-Bonnet formula. In fact, given a Riemann surface $(M^2,g)$, the Gauss curvature of $K_g>0$ everywhere if and only if there exists no geodesic triangle with total inner angle smaller than $\pi$. This fact is generalized by Alexandrov \cite{AleksandrovRerestovskiiNikolaev86} in the study of sectional curvature lower bounds in all dimensions. 

As a first step towards Conjecture \ref{conj:dihedral.rigidity}, Gromov studied the case for cubes, and obtained the following theorem:

\begin{theo}[\cite{Gromov2014Dirac}]\label{theo.cube.comparison.nonrigid}
Let $M=[0,1]^n$ be a cube, and $g$ be a Riemannian metric on $M$. Then $(M,g)$ cannot simultaneously satisfy:
\begin{enumerate}
    \item The scalar curvature of $g$ is positive;
    \item Each face of $M$ is strictly mean convex with respect to the outward normal vector field;
    \item Everywhere the dihedral angle between two faces of $M$ is acute.
\end{enumerate}
\end{theo}

The crucial observation is that conditions (2) and (3) may be interpreted as $C^0$ properties of the metric $g$. Thus, Gromov proposed a possible of definition of `$R\ge 0$' for $C^0$ metrics:
\begin{multline*}
    R(g)\ge 0 \Leftrightarrow \text{there exists no cube }M \\
    \text{ with mean convex faces and everywhere acute dihedral angle.}
\end{multline*}

Using Theorem \ref{theo.cube.comparison.nonrigid}, Gromov is able to deduce the following convergence result for scalar curvature. We remark that this result has been generalized by Bamler \cite{bamler2016ricci} using Ricci flow.

\begin{coro}
    Let $M^n$ be a smooth manifold, $g$, $g_k$, $k\ge 1$ be a sequence of smooth metrics on $M$. Suppose the scalar curvature of $g_k$ is nonnegative everywhere on $M$, and that $g_k\rightarrow g$ in $C^0$ as tensors. Then the scalar curvature of $g$ is also everywhere nonnegative.
\end{coro}

Gromov's proof (or at least a sketch of proof) of Theorem \ref{theo.cube.comparison.nonrigid} is based on an beautiful idea involving doubling cube $n$ times, and reduce the conjecture to the well-known fact that the $n$-dimensional torus admits no metric with positive scalar curvature. See \cite{LiMantoulidis2018positive} and \cite{Li2017polyhedron} for a detailed discussion. We note that a related technique has recently been explored by Kazaras \cite{Kazaras2019desingularizing} for the study of singular 4-manifolds. 

However, this idea leaves the following two questions open. First, it relies on the fact that the cube is the fundamental domain of $\ZZ^n$ action on $\RR^n$, hence is not applicable to general polytopes. Second, and perhaps more importantly, this argument cannot handle the rigidity statement in Conjecture \ref{conj:dihedral.rigidity}. In fact, even in the Euclidean space, it is unknown whether one can perturb a flat convex polyhedron, such that the faces are still minimal surfaces, while the dihedral angles between them remain the same. See section 1.5 of \cite{Gromov2014Dirac}.

In dimension $3$, Conjecture \ref{conj:dihedral.rigidity} is verified by the author \cite{Li2017polyhedron} for a large collection of polytopes, including the cubes and certain 3-simplices. The idea is to relate Conjecture \ref{conj:dihedral.rigidity} with a natural geometric variational problem of capillary type. The primary scope of this paper is to extend the idea in \cite{Li2017polyhedron} and prove Conjecture \ref{conj:dihedral.rigidity} for a general type of polyhedra, called \textbf{prisms}, of dimensions up to $7$. 

\begin{defi}\label{definition.euclid.prisms}
Let $n\ge 2$, and $P_0\subset \RR^2$ be an Euclidean polygon whose interior dihedral angles are all no larger than $\pi/2$. We call the polyhedron $P=P_0\times [0,1]^{n-2}$ a prism.
\end{defi}

Conjecturally the condition that all interior angles of $P_0$ do not exceed $\pi/2$ is technical- it guarantees sufficient regularity of solutions to certain elliptic equations, as will be seen from the paper. The primary objective of this paper are Riemannian polyhedra which admits a degree one map onto an Euclidean prism. Precisely, we make the following definition.

\begin{defi}\label{definition.riemann.prism}
We call a compact Riemannian manifold with boundary $(M,g)$ a polyhedron, if at every $x\in M$, there exists $r>0$ such that $(M\cap B_r(x),g)$ can be isometrically embedded to some $\RR^N$, and there is a diffeomorphism $\phi_x: B_r(x\in \RR^N)\to \RR^N$ with $\phi_x(B_r\cap M) = P\cap B_1(0^N)$ for some Euclidean polyhedal cone $P$. Further, we require that $\phi_x$ is $C^{2,\alpha}$ for some $\alpha\in (0,1)$ independent of $x$. 

If $P^n$ is an Euclidean prism and $M^n$ is a polyhedron, we call $M$ is over-prism of type $P$, if it admits a degree one map $\Phi$ onto $P$, such that for each integer $k\in [0,n]$ and on each $k$-face of $M$, the restriction of $\Phi$ is also a degree one map to a $k$-face of $P$.
\end{defi}

Some examples of over-prism polyhedrons include the interior connected sums of $P$ with a smooth manifold, as well as the manifold obtained by applying a surgery along a homotopically trivial knot in the interior of $P$ (see Section 3 of \cite{BoileauWang1996degree}). We now state the main result of this paper.

\begin{theo}\label{theo.dihedra.rigidity}
Let $2\le n\le 7$, $P^n$ be an Euclidean prism. Assume $M^n$ is a polyhedron, over-prism of type $P$. Then Conjecture \ref{conj:dihedral.rigidity} holds for $M$. Precisely, if $(M,g)$ is a Riemannian polyhedron such that
\begin{enumerate}
    \item The scalar curvature of $g$ is nonnegative;
    \item Each face of $M$ is weakly mean convex;
    \item The dihedral angles between adjacent faces of $M$ is everywhere equal to the corresponding (constant) dihedral angles of $P$.
\end{enumerate}
Then $(M,g)$ is isometric to an Euclidean prism.
\end{theo}

In \cite[Section 11]{Gromov2018Metric}, Gromov claimed and sketched a so-called bending construction on Riemannian polyhedrons. In particular, suppose $(M,g)$ is a Riemannian polyhedron, over-prism of type $P$, and $\Phi: M \to P$ is a degree one polyhedral map. One may deform $g$ to $\tilde g$ by bending the faces inward in a neighborhood of the edges, such that $(M,\tilde g)$ satisfies that:
\begin{itemize}
	\item $R(\tilde g)\ge R(g)$ in $M$;
	\item $H(\tilde g)\ge H(g)$ on each face of $M$;
	\item The dihedral angles of $(M,\tilde g)$ is everywhere equal to those of $P$.
\end{itemize}
See more discussions in Appendix \ref{appendix.section.bending.construction}. Granted this construction, one may replace assumption (3) in the statement of Theorem \ref{theo.dihedra.rigidity} to the weaker one:
\begin{enumerate}
	\item[(3')] The dihedral angles between adjacent faces of $M$ is everywhere\textit{ less than or equal} to the corresponding (constant) dihedral angles of $P$.
\end{enumerate}

As a side remark, we can combine this bending construction with Theorem 1.4 and 1.5 in \cite{Li2017polyhedron}, and establish Conjecture \ref{conj:dihedral.rigidity} for $3$-dimensional cones and prisms. In particular,

\begin{theo}\label{theo.for.3.simplices}
	Conjecture \ref{conj:dihedral.rigidity} holds for all 3-dimensional simplices $P$, such that $P$ has a face $B$ where the three dihedral angles along $B$ are all non-obtuse.
\end{theo}

It would be interesting to contextualize Theorem \ref{theo.dihedra.rigidity} in the study of scalar curvature on smooth manifolds with boundary. Following the positive mass theorem by Schoen-Yau \cite{SchoenYau1979ProofPositiveMass} and Witten \cite{Witten1981newproof}, Shi-Tam \cite{ShiTam2002positivemass} established a comparison theorem for smooth convex regions in $\RR^3$. See also (\cite{Miao02,EichmairMiaoWang2012extension}). Precisely, let $\Omega\subset \RR^3$ is a convex domain, and $(\Omega_0,g)$ is smooth Riemannian manifold with $R(g)\ge 0$, such that $\partial\Omega$ and $\partial \Omega_0$ are isometric with induced metrics. Then
\begin{equation}\label{equation.intro.shi.tam}
    \int_\Omega H dA - \int_{\Omega_0} H_0 dA\ge 0,
\end{equation}
where $H, H_0$ are the mean curvatures of $\Omega$, $\Omega_0$, respectively. The quantity on the left in \eqref{equation.intro.shi.tam} is the Brown-York quasi-local mass of $\partial \Omega$. See \cite{LiuYau2006positivity,WangYau2009isometric} for more discussions. The deep phenomenon illustrated by these results is that for a smooth Riemannian manifold with boundary $(\Omega,g)$, $R(g)$ in $\Omega$ and $H(g)$ on $\partial \Omega$ are a pair of quantities that behave oppositely. 

Theorem \ref{theo.dihedra.rigidity} is an extension of the above observation to Riemannian polyhedron $M^n$. In fact, if we view $\partial M$ as a varifold, then its total first variation has support on its faces as well as on its edges (intersection of adjacent faces), but not on any lower dimensional singular strata of $\partial M$. Moreover, if $F_1,F_2$ are two adjacent faces, then along $E=F_1\cap F_2$, the singular mean curvature of $\partial M$ is equal to $(\pi-\measuredangle(F_1,F_2))\delta_{E}$. We therefore observe that Theorem \ref{theo.dihedra.rigidity} precisely states that one cannot increase the scalar curvature of an Euclidean prism, while increase its boundary mean curvature the same time.

We remark that dihedral angle deficit also relates to the positive mass theorem. In a recent paper \cite{miao2019measuring}, Miao gave an explicit formula of the ADM mass of an asymptotically flat 3-manifold in terms of the interior scalar curvature, boundary mean curvature and the dihedral angle deficit of large coordinate cubes. Notably, if the conditions (1)-(3) in Theorem \ref{theo.dihedra.rigidity} are all satisfied for large geodesic cubes, the mass is non-positive. This result is inspired by a previous work of Stern \cite{stern2019scalar} (see also \cite{bray2019harmonic,bray2019scalar}), and (for now) only works in 3 dimensions. In section 5, we will see that Theorem \ref{theo.dihedra.rigidity} gives a rapid proof of the positive mass theorem, thus may be regarded as a localized positive mass theorem. It would be interesting to further study the connection between Conjecture \ref{conj:dihedral.rigidity} and mass/quasi-local mass in future.

\subsection{An outline of the proof}
From now on we fix an Euclidean prism $P^n=P_0^2\times [0,1]^{n-2}$. Denote by $g_{Euclid}$ the Euclidean metric on $P$. Suppose $(M^n,g)$ is a Riemannian polyhedron, over-prism of type $P$. Let $\Phi:M\rightarrow P$ be the degree one map. Inspired by the classical dimension descent technique of Schoen-Yau, we perform induction on the dimension $n$. When $n=2$, the Gauss-Bonnet theorem implies that
\begin{equation}\label{equation.introduction.GB}
    \int_M K dA +\int_{\partial M} k_g ds + \sum_{V\text{ is a vertex}}(\pi-\measuredangle(V))=2\pi\chi(M).
\end{equation} 
Since $M$ admits a degree one map onto $P$ and the sum of exterior angles of $P$ is $2\pi$, we conclude that the sum of exterior angles of $M$ is at least $2\pi$. On the other hand, since $M$ is connected, $\chi(M)\le 1$. Thus, \eqref{equation.introduction.GB} implies that $K=0$ in $M$ and $k_g=0$ on $\partial M$. Hence $M$ is isometric to a planar polygon.

Now suppose $n\ge 3$. We denote 
\begin{equation*}
    \begin{split}
    F_T=\Phi^{-1}(P_0\times [0,1]^{n-3}\times \{1\}),\quad F_B=\Phi^{-1}(P_0\times [0,1]^{n-3}\times \{0\}),\\ 
    F_L=\partial M\setminus (F_T\cup F_B),    
    \end{split}
\end{equation*}
and call the faces $F_T,F_B,F_L$ top, bottom and lateral faces of $M$. Consider the variational problem
\begin{equation}\label{problem.variation}
    I=\inf_{\Omega\in \cC}\{\cH^{n-1}(\partial \Omega\cap \mathring{M})\}
\end{equation}
where $\Omega$ is taken from the collection
\begin{multline*}
    \cC=\bigg\{\Omega=\Phi^{-1}(\tilde{\Omega}): \tilde{\Omega} \text{ is an open set of } P, P_0\times [0,1]^{n-3}\times\{1\}\subset \tilde{\Omega},\\
    \left(\tilde{\Omega}\cap P_0\times [0,1]^{n-3}\right) \times\{0\}=\emptyset\bigg\}.
\end{multline*}

Let $\Sigma=\partial \Omega\cap \mathring{M}$. In other words, we consider the least area hypersurface $\Sigma$ separating the top faces $F_T$ with the bottom face $F_B$. By the Federal-Fleming compactness theorem, the minimizer of \eqref{problem.variation} is obtained by an integral current $\Omega$, and $\partial \Omega$ is also integral. Also, if $\Omega\in \cC$, then $\Phi|_{\partial \Omega\cap \mathring{M}}$ is a degree one map onto $P_0\times [0,1]^{n-3}$. Since current convergence preserves homology type, we conclude that the minimizing hypersurface $\Sigma$ also has a degree one map onto $P_0\times [0,1]^{n-3}$. Hence, $ (\Sigma,g)$ is a Riemannian polyhedron, over-prism of type $P_0\times [0,1]^{n-3}$.

A great challenge is to establish suitable regularity for the hypersurface $\Sigma$. Notably, the regularity theory of free boundary minimal surface in general non-smooth domains is far from being understood. In our setting, this is done in two steps. First, we extend a weak boundary maximum principle by Li-Zhou \cite{LiZhou2017maximum} for free boundary varifolds in smooth manifolds, to a strong maximum principle in our polyhedron. The conclusion is that the minimizer $\Sigma$ is either disjoint from $F_T$ and $F_B$, or entirely lies in $F_T$ or $F_B$. Then we extend the regularity theory developed in \cite{EdelenLi2019regularity} and prove that $\Sigma$ is a $C^{2,\alpha}$ hypersurface, when $n\le 7$.

The crucial observation is that, after a conformal deformation, $\Sigma$ will be a Riemannian polyhedron satisfying conditions (1)-(3) of Theorem \ref{theo.dihedra.rigidity}. By repeating the argument, we construct a slicing of $M$ by free boundary minimal hypersurfaces
\[\Sigma^2\subset \cdots\subset \Sigma^{n-1}\subset M\]
where each $\Sigma_j$ is a $j$-dimensional Riemannian polyhedron, over-prism of type $P_0\times [0,1]^{j-2}$, and conditions (1)-(3) are satisfied. By induction, each $\Sigma_j$ is isometric to a flat Euclidean prism.

Finally we perform the rigidity analysis. Rigidity phenomenon in scalar curvature has caught lots of attentions in recent years, see \cite{CaiGallowway2000rigidity,BrayBrendleNeves2010rigidity,ChodoshEichmairMoraru2018splitting} and the survey article \cite{Brendle2012rigidity}. We extend an elegant idea of Carlotto-Chodosh-Eichmair \cite{CarlottoChodoshEichmair2016effective} in the study of effective positive mass theorems, which was inspired on earlier works on deformations of metrics with other curvature conditions (see \cite{Ehrlich1976metricdeformation,AndersonRodriguez1989minimal,liu2013three-manifolds}). Precisely, through an argument involving a family of well-chosen local conformal deformations, we prove that $M$ contains a dense collection of free boundary area minimizing surfaces, whose boundaries are also dense on $\partial M$. This is enough to guarantee that $M$ is isometric to an Euclidean prism. To carry out the argument, we establish new curvature estimates of free boundary minimal hypersurfaces in a Riemannian polyhedron, inspired by previous work by Guang-Li-Zhou \cite{GuangLiZhou2016curvature}. 

\subsection{Some future perspectives}
We speculate that the minimal slicing strategy should be able to prove Conjecture \ref{conj:dihedral.rigidity} for more polytopes. Precisely, we define a class of $n$-dimensional Euclidean polyhedron $\cP_n$ as follows. Let $\cP_2$ be the collection of all convex polygons. Given $\cP_{n-1}$, defined $\cP_n$ be the set of $n$-dimensional Euclidean polyhedra $P$, such that $P$ has a topological foliation where all leaves are parallel and are similar to a fixed element $P_{n-1}\subset \cP_{n-1}$ \footnote{That is, each leaf is congruent up to scaling, but not rotation, to an $(n-1)$ dimensional polyhedron $P_{n-1}$.}. In particular, the prisms in Definition \ref{definition.euclid.prisms} as well as all $n$-dimensional simplices are in $\cP_n$. We now describe a heuristic argument that would imply conjecture \ref{conj:dihedral.rigidity} for a polyhedron $P\in \cP_n$.

For the sake of simplicity, let us assume $P$ is a simplex. Fix a vertex $v$ of $P$. Let $F_B$ be the opposite face of $v$ in $P$. Denote other faces of $P$ by $\{F_j\}_{j=1}^k$, and the interior angle between $F_j$ and $F_B$ by $\gamma_j$. \footnote{In the case of a prism, each $\gamma_j=\pi/2$.} Consider the variational problem
\begin{equation}\label{problem.variation.capillary}
    I=\inf\left\{\cH^{n-1}(\partial \Omega\cap \mathring{M})-\sum_{j=1}^k\cos\gamma_j \cH^{n-1}(\partial \Omega\cap \Phi^{-1}(F_j))\right\}
\end{equation}
for Caccioppoli sets $\Omega$ of $M$, in the class $\cC$ as before. The minimizer of \eqref{problem.variation.capillary} $\Omega$ gives a capillary minimal hypersurface $\Sigma=\Omega\cap \mathring{M}$. Inductively one obtains a slicing of $M$ by capillary minimal hypersurfaces
\[\Sigma^2\subset\cdots\subset\Sigma^{n-1}\subset M\]
and thus by applying the very same geometric argument as this paper, one proves Conjecture \ref{conj:dihedral.rigidity}.

The entire difficulty is the lack of regularity theory for capillary surfaces. De Philippis-Maggi \cite{DePhilippisMaggi15regularity} has the best partial regularity for capillary energy function so far, which states that when the domain is $C^{1,1}$, energy minimizing capillary hypersurface has singularities of codimension at least $2$. Conjecturally this is not sharp. We also note that the recent progress by Schoen-Yau \cite{SchoenYau2017positive} suggests that one may be able to carry out the slicing argument with presence of singularities. These are interesting questions to investigate in future.

\textbf{Acknowledgement:} The author wishes to thank Rick Schoen, Brian White, Otis Chodosh and Nick Edelen for various helpful conversations that are important for several stages of this project, as well as Misha Gromov, Fernando Marques and Michael Eichmair for their interests in this work. Special thanks go to Christina Sormani for organizing \textit{Emerging Topic on Scalar Curvature Seminar}, from which the author learned a lot. Part of this project was developed at the Institute for Advanced Study during the special year\textit{ Variational Methods in Geometry}. The author would like to acknowledge the hospitality of the IAS and the excellent working conditions.

\section{Preliminaries of Riemannian polyhedron and free boundary minimal surfaces}
In this section we review some preliminary facts on the Riemannian geometry in a polyhedron, and on free boundary minimal surface in domains with piecewise smooth boundary. We also include a regularity theorem for free boundary minimal hypersurfaces in polyhedron domains.

\subsection{Local geometry of a Riemannian polyhedron}
Let $M^{n}$ be a Riemannian polyhedron such that any two adjacent faces meet at constant angle along their intersection. Let $p\in \partial M$ lying on $k$ faces, $F_i$, $i=1,\cdots,k$. Thus, a local neighborhood of $p$ in $M$ is diffeomorphic to the region $\{(x_1,\cdots,x_n): x_1^2+\cdots+x_n^2<r^2, x_j\ge 0, j=1,\cdots,k\}$, where under this diffeomorphism, the face $F_j$ got mapped onto $\{x_j=0\}$. Denote the interior angle between $F_i$ and $F_j$ by $\gamma_{ij}$. We recall the following basic facts of Riemannian geometry, whose proof is a simple generalization of Lemma 2.1 in \cite{LiZhou2017maximum}.

\begin{lemm}\label{lemma.local.foliation.1}
    There exists a constant $\delta>0$, a neighborhood $U$ of $p$ in $M$, and a foliation $F_1^{s}$ with $s\in [0,\delta)$, of $U$, such that $F_1^{0}=F_1\cap U$, and each $F_1^s$ meets $F_j$ at constant angle $\gamma_{1j}$, $j=1,\cdots,k$.
\end{lemm}

By adapting Lemma \ref{lemma.local.foliation.1} to each face $F_j$, $j=1,\cdots,k$, and possibly shrink the neighborhood, we find, for each $j=1,\cdots,k$, a foliation $\{F_j^s\}_{s\in [0,\delta)}$, such that $F_j^0=F_j\cap U$, and for each pair $i\ne j$, $F_i^s$ meets $F_j$ at constant angle $\gamma_{ij}$. For each $j=1,\cdots,k$ and a point $q\in U$, define the function $x_j(q)=s$ if and only if $q\in F_j^s$. Then extend $\{x_j\}_{j=1}^k$ to $\{x_j\}_{j=1}^n$, such that the latter gives a diffeomorphism of $U$ to a region in $[0,\infty)^k\times \RR^{n-k}$. Therefore, we have the following lemma, which is a slight generalization to the existence of local Fermi coordinates on smooth manifolds with boundary.

\begin{lemm}\label{lemma.local.foliation.2}
    Let $(M,g)$ be a Riemannian polyhedron whose adjacent faces meet at constant angle along their intersection. For any $p\in \partial M$ where $p$ lies on $k$ different faces $F_1,\cdots,F_k$, there is a local coordinate system $\{x_j\}_{j=1}^n$, such that $F_j$ is given by $\{x_j=0\}$, $j=1,\cdots,k$ and the metric satisfies $g(\partial_i,\partial_j)=-\cos \gamma_{ij}$, $g(\partial_i,\partial_i)=1$ on $F_i$, for $i\ne j$, $1\le i,j\le k$.
\end{lemm}

\subsection{Preliminaries of free boundary minimal surfaces}
We briefly describe the geometry of free boundary minimal surfaces. Let $(M,g)$ be an $n$-dimensional Riemannian polyhedron, and $g$ is a $C^{2,\alpha}$ metric. Suppose $(\Sigma,\partial \Sigma)\subset (M,\partial M)$ is an embedded hypersurface. We say $(\Sigma,\partial \Sigma)$ is an embedded free boundary minimal hypersurface, if the interior of $\Sigma$ is minimal (having zero mean curvature), and $\Sigma$ meets $\partial M$ orthogonally on the smooth part of $\partial M$.

Free boundary minimal hypersurfaces are the critical points of $(n-1)$-dimensional volume functional of $(M,g)$ among class of all embedded hypersurfaces. Given a vector field $Y$ in $M$ which is tangential to $\partial M$, let $\psi_t$ be the one-parameter family of diffeomorphisms generated by $Y$. Then for $\ep$ small enough, $\psi_t(\Sigma)$ is a smoothly embedded hypersurface in $M$. The first variational formula implies that
\begin{equation}\label{formula.first.variation.regular}
    \td{}{t}\bigg\vert_{t=0}\cH^{n-1}(\psi_t(\Sigma))=-\int_\Sigma X\cdot \vec{H} dV+ \int_{\partial \Sigma} X\cdot \eta dS,
\end{equation}
where $\vec{H}$ is the mean curvature vector field of $\Sigma$, $\eta$ is the outward conormal vector field of $\partial \Sigma\subset \Sigma$, $dV$ and $dS$ are the induced volume forms on $\Sigma$ and $\partial \Sigma$, respectively. It is then clear that $\Sigma$ is a critical point of the $(n-1)$-dimensional volume functional, if and only if \eqref{formula.first.variation.regular} vanishes for all admissible vector fields $Y$, which is equivalent to the fact that $\vec{H}\equiv 0$ and that $\Sigma$ meets $\partial M$ orthogonally.

The second variational formula for a free boundary minimal hypersurface is given as
\begin{multline}\label{formula.second.variation}
    \td{^2}{t^2}\bigg\vert_{t=0} \cH^{n-1}(\psi_t(\Sigma))=Q(f,f)\\
    =\int_\Sigma |\nabla_\Sigma f|^2-(|A_\Sigma|^2+\Ric_M (\nu,\nu))f^2 dV-\int_{\partial \Sigma} \secondfund_{\partial M} (\nu,\nu)f^2 dA.
\end{multline}
Here $\nu$ is the unit normal vector field of $\Sigma$ in $M$, $f=Y\cdot \nu$ is the normal component of the variation, $A_\Sigma$ is the second fundamental form of $\Sigma$ in $M$, $\secondfund$ is the second fundamental form of $\partial M$ in $M$, $\Ric_M$ is the Ricci curvature of $M$. We have adopted the sign convention that $A_\Sigma>0$ for convex hypersurfaces in $\RR^n$. In particular, the unit $2$-sphere in $\RR^3$ has constant mean curvature $2$.

Call a two-sided free boundary minimal hypersurface $\Sigma$ stable, if its second variation is always nonnegative. By \eqref{formula.second.variation}, $\Sigma$ is stable, if and only if $Q(f,f)\ge 0$ for any smooth function $f$.

\subsection{Regularity of free boundary area minimizing currents}
In \cite{EdelenLi2019regularity}, the authors established a regularity theory for free boundary varifolds in locally convex domains. We briefly describe the results relevant to this paper. Given an integer $2\le k\le n-1$, let $C_0=W_0^2\times [0,\infty)^{k-2}$ be a polyhedral cone, where $W_0\subset \RR^2$ is a nob-obtuse wedge region
\[W_0=\{(r,\theta): r\ge 0, 0\le \theta\le \theta_0\}\]
in polar coordinates, here $\theta_0\in (0,\pi/2]$. Notice that $C_0$ is the tangent cone of a convex Riemannian polyhedron at a boundary point. Suppose $\Phi:B_1(0^n)\rightarrow \RR^{n}$ be a $C^{2,\alpha}$ mapping with 
\[\Phi(0)=0, \quad D\Phi(0)|_{0}=Id,\quad |\Phi-Id|_{C^{2,\alpha}(B_1)}\le \Gamma\le 1.\]
Let $C=\Phi(C_0\times \RR^{n-k})$. 

Assume $\Sigma$ be an area minimizing current in $C$, such that $0\in\spt \Sigma$, and $\Sigma$ has free boundary, in the sense that $\delta\Sigma(X)=0$ for all vectors $X\in C_c^1(B_1,\RR^n)$ which are tangential to $C$. We then have the following theorem:

\begin{theo}\label{theo.regularity.of.free.boundary.minimizing.surface}
    Suppose $\Sigma$ is an area minimizing current in $C$, $0\in\spt \Sigma$. There are constants $\alpha_0(\theta_0)$, $\delta(\theta_0,n)$, $\rho\in (0,1)$, and some linear subspace $\RR^{n-k-1}\subset \RR^{n-k}$, such that for any $\alpha\in (0,\alpha_0)$, if $\Gamma\le \delta^2$, then there exists a $C^{1,\alpha}$ function $u: (C_0\times \RR^{n-k-1})\cap B_{\rho}(0)\rightarrow \RR$, so that
    \[\Sigma\cap B_{\rho}=[\graph u].\]
    Moreover, we have the following estimate:
    \begin{equation}\label{theo.regularity.quantitative.estimates}
        |Du(x_1)-Du(x_2)|\le \left(\frac{|x_1-x_2|}{\rho}\right)^\alpha.
    \end{equation}
\end{theo}

Theorem \ref{theo.regularity.of.free.boundary.minimizing.surface} is proved in \cite{EdelenLi2019regularity} via an Allard type theorem for free boundary varifolds:

\begin{theo}\label{theo.regularity.allard.type}
    Suppose $V$ is a free boundary stationary varifold in $C$, $0\in \spt V$. There are constants $\ep(\theta_0,n)$, $\alpha_0(\theta_0)$, $\rho \in (0,1)$, such that if
    \[\|V\|(B_1(0))\le (1+\ep)\cH^{n-1}((C_0\times \RR^{n-k-1})\cap B_1(0)),\]
    then $V\cap B_\rho(0)$ is given by the graph of a $C^{1,\alpha}$ function $u$, for all $\alpha\in (0,\alpha_0)$, with the estimates \ref{theo.regularity.quantitative.estimates} holds.
\end{theo}

We refer the readers to \cite{EdelenLi2019regularity} for a detailed proof.

\section{Ancillary estimates for free boundary area minimizing hypersurfaces}

\subsection{A strong maximum principle for varifolds with free boundary}\label{section.free.boundary.minimal.slicing}

We observe that \eqref{problem.variation} is a problem with barriers, since we require that $F_T\subset \Omega$ and $F_B\cap \Omega=\emptyset$. In general, minimization problem with barriers do not necessarily produce a minimal surface. However, in our setting, since both barriers, $F_T$ and $F_B$, are mean convex hypersurfaces, along which the dihedral angle of $M$ is everywhere $\pi/2$, any varifold  that is stationary for \eqref{problem.variation} should be disjoint from both $F_T$ and $F_B$, and hence is a minimal hypersurface $\Sigma$. This is achieved via the following varifold maximum principle. 

\begin{theo}\label{theo.strong.maximum.principle}
    Let $\Sigma=\Omega\cap \mathring{M}$ be a varifold, stationary in the class $\cC$ for $\eqref{problem.variation}$. Assume that $F_T$ and $F_B$ are weakly mean convex. Then either $\spt \Sigma$ is disjoint from the closure of $F_T$ and $F_B$, or $\Sigma$ lies entirely in $F_T$ or $F_B$.
\end{theo}

Before embarking into the proof, we remark that analogous maximum principles have been established in various settings. If $\Sigma$ is a $C^2$ hypersurface, this is just interior maximum principle and boundary Hopf lemma for elliptic equations. For varifolds without boundary, a strong maximum principle in codimension one was proved by Solomon-White \cite{SolomonWhite1989strong}. White \cite{white2010maximum} then generalized the theorem to arbitrary codimension. Similar statements hold if the ambient manifold has various types of singularities. See, e.g. \cite{Simon87strict,Wickramasekera14strongmaximum}. For free boundary varifolds, a weak maximum principle was done by Li-Zhou \cite{LiZhou2017maximum}. Theorem \ref{theo.strong.maximum.principle} is an extension of \cite{LiZhou2017maximum} when the ambient manifold is not necessarily smooth. Also, the strong maximum principle for stationary varifold is previously unknown, even in smooth ambient manifolds. Our proof is inspired by \cite{SolomonWhite1989strong} and \cite{LiZhou2017maximum}.

As a consequence of Theorem \ref{theo.strong.maximum.principle}, the minimizer of \eqref{problem.variation} gives an area minimizing hypersurface with free boundary along $F_L$. Combined with Theorem \ref{theo.regularity.of.free.boundary.minimizing.surface}, we have know $\Sigma$ is a $C^{1,\alpha}$ hypersurface. Further more, with the angle assumption in \ref{definition.euclid.prisms}, we can upgrade $C^{1,\alpha}$ to $C^{2,\alpha}$ by Appendix B.

\begin{theo}\label{theo.regularity}
    Let $\Omega$ be the minimizer of \eqref{problem.variation}. Then $\Sigma=\Omega\cap \mathring{M}$ is a $C^{2,\alpha}$ hypersurface up to its corners.
\end{theo}

We first prove the weak maximum principle as follows.
\begin{prop}\label{proposition.weak.maximum.principle}
    In addition to the assumptions of Theorem \ref{theo.strong.maximum.principle}, assume that $F_T$, $F_B$ are strictly mean convex. Then $\spt \Sigma\cap F_T=\spt \Sigma\cap F_B=\emptyset$.
\end{prop}

\begin{proof}
    It suffices to prove that $\spt \Sigma\cap F_B=\emptyset$. Assume, for the sake of contradiction, that there exists $p\in \spt \Sigma\cap F_B$. By the maximum principle \cite{SolomonWhite1989strong}, $p\in \partial F_B$. We construct a vector field $X$ supported in a neighborhood of $p$, tangential to $F_L$, and points into $M$ along $F_B$, such that $\delta \Sigma(X)<0$. We proceeds as in \cite{LiZhou2017maximum}. Suppose a neighborhood $U$ of $p$ is diffeomorphic to $\{(x_1,\cdots,x_n): |x|\le r, x_j\ge 0, j=1,\cdots, k\}$, where $F_B\cap U$ lies on $\{x_1=0\}$, $F_j\cap U$, $2\le j\le k$, lies on $\{x_j=0\}$, and the coordinate functions $\{x_i\}_{i=1}^n$ are given by Lemma \ref{lemma.local.foliation.2}. Now extend the metric tensor smoothly into $B_r(0)\subset \RR^n$, such that $\{x_1=0\}$ is still normal to other faces. Denote $M^*$ the region given by $\{(x_1,\cdots,x_n): x_j\ge 0, j=2,\cdots,k\}$. With slight abuse of notation, we denote the neighborhood of $p$ in this larger Riemannian manifold by $U$.
    
    For $\ep>0$ sufficiently small enough, define
    \[\overline F_B =\{x_1= \ep (x_1^2+\cdots + x_n^2)^2 \}.\]
    
    Notice that $\overline{F}_B$ and $F_B$ touches in second order at $p$, and meets $F_j$, $2\le j\le k$, orthogonally. Take the foliation $\overline{F}_B^s$ given by Lemma \ref{lemma.local.foliation.1} such that $\overline{F}_B^s$ is orthogonal to $F_j$,  and define $s$ to be the function in $U$ such that $s(q)$ is the value for which $q\in \overline{F}_B^{s(q)}$. Since $F_B$ is strictly mean convex and agrees to $\overline F_B$ in second order, the mean curvature of $\overline F_B^s$ is lower bounded by $c_0>0$ for $s\in [0,\ep]$, by shrinking $\ep>0$ if necessary. Then $\nabla s=\psi \nu$ for some function $\psi$, where $\nu$ is the unit normal vector field of $\overline{F_B}^{s(q)}$. Note that $\psi(p)=1$, and $\nu$ is tangential to $F_j$, $j=2,\cdots,k$.
    
    Define a vector field $X$ on $M^*$ by letting $X(q)=\phi(s(q))\nu(q)$, where $\phi(s)$ is a cutoff function defined by $\phi(s)=e^{1/(s-\ep)}$ when $0\le s\le \ep$, $\phi= 0$ when $s\ge \ep.$ In $U$, $X$ is a vector field tangential to $\partial M$ along $F_j$, $j\ge 2$, and is inward pointing on $F_B$, provided that $\ep$ is sufficiently small.
    
    Near $p$, take a new coordinate system $\{x_j\}_{j=1}^n$ in the polyhedron enclosed by faces $\overline{F}_B,F_2,\cdots,F_n$, given by Lemma \ref{lemma.local.foliation.2}. Let $\{e_i=\partial_i/|\partial_i|\}$ be the unit frame. Note that at $p$, $\bangle{e_1,e_j}=0$, $j\ge 2$. Define a bilinear form $Q$ on $T M^*$ by letting $Q(u,v)(q)=\bangle{\nabla_u X,v}(q)$. We calculate the components of $Q$. By straightforward calculation, we have that $Q(e_1,e_1)=\phi'\psi$, and for $i,j\ge 2$, 
    $Q(e_i,e_j)=-\phi A^{\overline{F}_B}(e_i,e_j), Q(e_i,e_1)=0, Q(e_1,e_i)=\phi\bangle{\nabla_\nu\nu, e_i}.$
    Here $A^{\overline{F}_B}$ is the second fundamental form of the hypersurface $\overline{F}_B$ with respect to $\nu$.


    Since $\overline{F}_B$ is a small $C^2$ perturbation of $F_B$, by taking $\ep\in (0,\ep_0)$ and the neighborhood $U$ sufficiently small, we have $|\bangle{\nabla_\nu \nu,e_i}|,|A^{\overline{F}_B^s}|<K$ for some constant $K$ depending only on $(M,g)$. In particular, for small $\ep$ and $U$, and $2\le i,j\le n$,
    $|Q(e_i,e_j)|,|Q(e_i,e_1)|,|Q(e_1,e_i)|\le K$. Also, by construction of $\phi$, it is straightforward to check that $\phi'\le -\frac{1}{\ep^2}\phi$. Finally, since $\psi(p)=1$, by possibly taking $U$ sufficiently small, we have that $\psi\ge \frac{1}{2}$ in $U$.
    
    We verify that for any $q\in U$, any $(n-1)$-dimensional subspace $P\subset T_q M$, $\tr_P Q<0$. Let $c_0>0$ be a lower bound of the mean curvature of $\overline{F}_B^s$ for $s\in [0,\delta)$. If $P=T_q \overline{F}_B^s$, then $\tr_P Q\le -c_0<0$. If $P\not\subset T_q^{\overline{F}_B^s}$, since $P$ and $T_q\overline{F}_B^s$ are two $(n-1)$ dimensional subspaces of $T_q M\simeq \RR^n$, there is an orthonormal basis $v_1,\cdots,v_{n-1}$ of $P$, such that $v_1,\cdots,v_{n-2}\in T_q \overline{F}_B^s$, $v_{n-1}\not\in T_q\overline{F}_B^s$. Take orthogonal decomposition $v_{n-1}=\cos \theta v_0+\sin\theta \nu$, $v_0\in T_q \overline{F}_1^s$ and $v_0\perp v_j$, $j=1,\cdots,n-2$. In particular, $v_0,\cdots,v_{n-2}$ is an orthonormal basis of $T_q \overline{F}_B^s$, and $v_0\perp \nu$. We express $v_0$ as a linear combination of $\{e_2,\cdots,e_n\}$: $v_0=\sum_{j=2}^n a_j e_j$. Notice that $\{e_j\}_{j=2}^n$ is not orthonormal in $U$, but $\bangle{e_i,e_j}(p)=\cos\gamma_{ij}$. So if we take the neighborhood $U$ small enough, $|a_j|<2$, for each $j=2,\cdots,n$. 
    \begin{equation}
        \begin{split}
            \tr_P Q &= \sum_{i=1}^{n-1} Q(v_i,v_i)= \sum_{i=0}^{n-1} Q(v_i,v_i) +\sin^2\theta (Q(\nu,\nu)-Q(v_0,v_0)) \\
                    &\qquad\qquad    +\sin\theta\cos\theta (Q(\nu,v_0)+Q(v_0,\nu))\\
                    &= -\phi H^{\overline{F}_B^s} + \sin^2\theta \left(\phi'\psi+\phi A^{\overline{F}_B^s}(\nu_0,\nu_0)\right)+\sin\theta\cos\theta Q(\nu,v_0)\\
                    &\le -c_0 \phi +\sin^2\theta \phi \left(-\frac{1}{2\ep^2}+K\right)+\sin\theta\cos\theta \sum_{j=2}^n a_j Q(\nu,e_j)\\
                    &\le -c_0 \phi+ \phi\left(-\frac{\sin^2\theta}{2\ep^2} +K\sin^2\theta+2(n-1)K|\sin\theta\cos\theta|\right)\\
                    &\le -c_0\phi+\phi\left(-\frac{1}{2\ep^2}+K+\frac{2(n-1)^2 K^2}{c_0}\right)\sin^2\theta+\phi\frac{c_0}{2}\cos^2\theta\\
                    &\le -\frac{c_0}{2} \phi,
        \end{split}
    \end{equation}
    for all $\theta \in [0,2\pi]$, provided we take $\ep\in (0,\ep_0)$ and $\ep_0>0$ is small enough (depending on $K,n,c_0$). This shows that $\tr_P Q<0$ for all $(n-1)$ dimensional subspace $P$. In particular, $\delta V (X)<0$ whenever $\spt V\cap U\ne \emptyset$. This finishes the proof.

\end{proof}

We proceed to the proof of Theorem \ref{theo.strong.maximum.principle}. The idea here is greatly inspired by the strong maximum principle for elliptic equations: assuming that $F_B,F_T$ are only weakly mean convex, if $\Sigma$ touches, say $F_B$, but do not entirely coincide with it, then we can deform $F_B$ to a strictly mean convex surface $\tilde{F}_B$, violating Proposition \ref{proposition.weak.maximum.principle}. This has been carried out for varifolds without boundary in \cite{SolomonWhite1989strong}. Suppose $p\in \spt\Sigma\cap \partial F_B$, and in a neighborhood $U$ of $p$ in $M$, there is a diffeomorphism $\Phi: \{(x_1,\cdots,x_n):x_j\ge 0,j=1,\cdots,k\}\rightarrow U$ with $\Phi(\{x_1=0\})\subset F_B$. As in the proof of Proposition \ref{proposition.weak.maximum.principle}, we extend the metric $g$ to $\{(x_1,\cdots,x_n): x_j\ge 0,j=2,\cdots,k\}$, such that $F_B$ is still orthogonal to the lateral faces $\Phi(\{x_j=0\})$, $j=2,\cdots,k$. For the rest of the argument, we describe the open set $U$ by the coordinates $\{x_j\}_{j=1}^n$ via $\Phi$.

For regularity purposes, we define the following domains, over which we will construct our free boundary hypersurface:
\[D_r=\left\{(x_2,\cdots,x_n): x_2,\cdots,x_k\ge 0, \sum_{j=2}^k (x_j+\frac{r}{n-1})^2+\sum_{j=k+1}^n x_j^2 \le r^2\right\}.\]
And similarly,
\[\Gamma_r=\left\{(x_2,\cdots,x_n): x_2,\cdots,x_k\ge 0, \sum_{j=2}^k (x_j+\frac{r}{n-1})^2+\sum_{j=k+1}^n x_j^2 =r^2\right\},\]
and $T_r$ be the set inside $B_r(0)$, where at least one of $x_j$, $j=2,\cdots,k$, is equal to $0$. Notice that $\Gamma_r$ meets each lateral face $\{x_j=0\}$, $2\le j\le n$, at an acute interior angle. We then define the space $C_0^{2,\alpha}(D_r)$ to be functions $u\in C^{2,\alpha}(D_r)$ such that $u=0$ on $\Gamma_r$. For $u\in C^{2,\alpha}(D_r)$, let $F_u$ be the hypersurface defined by $x_1=u(x_2,\cdots,x_n)$ in $M$. We prove the following foliation result:

\begin{lemm}\label{lemma.existence.local.foliation}
    There exists $\ep_0>0$, such that for any $r,s\in (-\ep_0,\ep_0)$, $f\in C^{2,\alpha}(D_r)$ with $|f|_{C^{2,\alpha}(D_r)}<\ep_0$, and $t\in [-r/2,r/2]$, there exists a function $u=u_{r,s,f,t}\in C^{2,\alpha}(D_r)$, such that $u=f+t$ on $\Gamma_r$, and the hypersurface $F_{u}$ is a free boundary hypersurface in $M$ whose mean curvature is $s$ with respect to the outward unit normal. Moreover, for each fixed choice of $r,s,f$, $\{F_u\}_{t\in [-r/2,r/2]}$ is a gives a foliation of a neighborhood of $p$.
\end{lemm}

\begin{proof}
    We use the implicit function theorem. Consider the rescaled manifold $M_r=r^{-1}(M-p)$. Clearly $M_r$ is a Riemannian polyhedron, and we can define the corresponding domains $D_1,\Gamma_1$, and the diffeomorphism $\Phi:D_1\rightarrow M_r\cap B_1(0)$ analogously. Define a map
    \[h:\RR\times \RR\times C_0^{2,\alpha}(D_1)\times C^{2,\alpha}(D_1)\rightarrow C^{0,\alpha}(D_1)\]
    by letting $h(r,t,v,f)=H_{v+f+t}-sr$, where $H_{v+f+t}$ is the mean curvature of $F_{v+f+t}$ with respect to outward unit normal. Also,
    \[\xi: \RR\times \RR\times C_0^{2,\alpha}(D_1)\times C^{2,\alpha}(D_1)\rightarrow C^{1,\alpha}(T_1)\]
    by letting $\xi(r,t,v,f)=\bangle{\eta_{M},\eta_{F_{v+f+t}}}$, where $\eta_M, \eta_{F_{v+f+t}}$ are the outward unit conormal of $M$, $F_{v+f+t}$ along $T_1$, respectively. Finally, define $\Theta = h\times \xi$.
    
    It is checked in the appendix of \cite{white1987space} that $\Theta$ is a $C^1$ map between the Banach spaces. Since when $r\rightarrow 0$, $M_r$ converges in $C^{2,\alpha}$ Cheeger-Gromov sense to an Euclidean polyhedron, hence we have the following linearized operator:
    \[D_4(0,t,0,0)(v)=Lv=\left(\Delta v, \pa{v}{\eta},\right).\]
    Clearly $v\mapsto Lv$ has trivial kernel in $C_0^{2,\alpha}(D_1)$. Notice that $D_1=B_1(0)\cap (W_2\times [0,\infty)^{k-2}\times \RR^{n-k-2})$ is an Euclidean polyhedron, where we assumed that the dihedral angle of $W_2$ is less than or equal to $\pi/2$. By the regularity theory developed in Appendix B, choosing $\alpha\in (0,1)$ properly,
    \[\|v\|_{2,\alpha,D_1}\le C(\|\Delta v\|_{0,\alpha,D_1}+\|v_\eta\|_{T_1}).\]
    
    Thus, by the implicit function theorem, for each $t\in [-1/2,1/2]$ and sufficiently small $r$ and $\|f\|_{2,\alpha,D_1}$, there exists a function $v=v_{r,f,t}$ such that $F_v$ meets $\partial M$ orthogonally and has mean curvature $sr$. Define $u_{r,s,f,t}=v_{r,f,t}+f+t$. Note that when $r=0$, $f=0$, $u=t$ for each $t\in [-1/2,1/2]$. Since $\Theta$ is a $C^1$ map, $\pa{u}{t}>0$ for small $r$, $f$. Thus $F_u$ gives a foliation in a neighborhood of $p$.
\end{proof}

Now we prove Theorem \ref{theo.strong.maximum.principle}.

\begin{proof}[Proof of Theorem \ref{theo.strong.maximum.principle}]
    It suffices to prove that if  $\spt \Sigma$ contains a point $p\in F_B$, it contains a neighborhood of $p$ in $F_B$. By \cite{SolomonWhite1989strong}, we only need to consider the case when $p\in \partial F_B$. For $r$ sufficiently small, consider a small neighborhood $\Phi(D_r\times [-r,r])$ of $p$. Assume, for the sake of contradiction, that $\spt \Sigma$ and $F_B$ does not coincide. Then there is $r$ sufficiently small, such that $\Gamma_r\times \{0\}$ is not entirely contained in $\spt \Sigma$. Take a function $f\in C^{2,\alpha}(\Gamma_r)$, $f\ge 0$ (but not identically $0$), and such that $\spt f\cap \spt \Sigma=\emptyset$. Now extend $f$ to $C^{2,\alpha}(D_r)$, and by slight abuse of notation, we denote the extension also by $f$. 
    
    Taking $r$, $\|f\|_{2,\alpha,D_r}$ small, we can apply Lemma \ref{lemma.existence.local.foliation} and find a local foliation $F_{u_{r,s,f,t}}$. Fix a choice of $r$. Since $\spt \Sigma$ lies strictly above $F_B$ over $\Gamma_r\cap \spt f$, by possibly replacing $f$ with $\ep f$ where $\ep>0$ is small enough, $\spt \Sigma$ also lies above $F_{u_{r,s,f,0}}$ over $\Gamma_r$. Fix this choice of $f$. Let $u_{s,t}=u_{r,s,f,t}$. Since $u_{0,0}\ge 0$ (but not identically 0) on $\Gamma_r$, by the Hopf maximum principle for minimal surfaces, $u_{0,0}>0$ in $D_r\setminus \Gamma_r$. In particular, $u_{0,0}(0)>0$. Now fix a choice of $s>0$ such that $u_{s,0}(0)>0$. 
    
    Consider the foliation $\{F_{u_{r,s,f,t}}\}_{t\in [-1/2,1/2]}$ for $r,s,f$ chosen as above. Let $u_t=u_{r,s,f,t}$. We see that each leaf $F_{u_t}$ is mean convex with respect to the outward unit normal, and $F_{u_0}$ lies above $\spt \Sigma$ over $p$, and at the same time, lies below $\spt \Sigma$ over $\Gamma_r$. Define $t_0$ be the smallest value of $t$ for which $F_{u_t}$ intersects $\spt \Sigma$. Then $t_0<0$. We observe that $\spt \Sigma$ and $F_{u_{t_0}}$ has an intersection in $D_r\setminus \Gamma_r$, but $F_{u_{t_0}}$ is strictly mean convex. This contradicts Proposition \ref{proposition.weak.maximum.principle}.
    
\end{proof}

\subsection{Curvature estimates for free boundary area minimizing hypersurfaces}
We now establish curvature estimates for free boundary minimizing hypersurfaces in a Riemannian polyhedron. This will be used in a compactness argument later in the proof of Theorem \ref{theo.dihedra.rigidity}. The proof is a rescaling argument together with the regularity theory in Theorem \ref{theo.regularity.of.free.boundary.minimizing.surface}.

\begin{theo}\label{theo.curvature.estimate}
    Suppose $(M,g)$ satisfies the same assumptions as in Theorem \ref{theo.dihedra.rigidity}, $(\Sigma,\partial \Sigma)\subset (M,\partial M)$ is a properly embedded free boundary area minimizing hypersurface. Then
    \[\sup |A_\Sigma|\le C(M,g),\]
    where $C>0$ is a constant depending only on $(M,g)$. Moreover, for a compact family of choices of $C^{2,\alpha}$ metrics $g$, the constant $C$ can be chosen uniformly.
\end{theo}

We separate the proof into a few steps. Recall that the rescale limit of $M$ near a boundary point takes the form $C=W_0\times [0,1)^{k-2}\times \RR^{n-k}$, here $W_0\subset \RR^2$ is a wedge region with opening angle not larger than $\pi/2$. Recall also that $C_0=W_0\times [0,1)^{k-2}$.

\begin{lemm}\label{lemma.curvature.estimates.local.epsilon}
    Suppose $n\le 7$. For any $\ep>0$, there exists a constant $\beta\in (0,1)$ such that the following holds: for any free boundary area minimizing hypersurface $\Sigma\subset C\cap B_1(0)$, $0\in \Sigma$,
    \begin{equation}\label{equation.smallness.area.excess}
        \cH^{n-1}(\Sigma\cap B_\beta(0))\le (1+\ep)\cH^{n-1}((C_0\times \RR^{n-k-1})\cap B_\beta(0)).
    \end{equation}
\end{lemm}

\begin{proof}
    Suppose the contrary, that there exists a sequence of free boundary area minimizing hypersurfaces $\Sigma_j\subset C\cap B_1(0)$, $0\in \Sigma_j$, and a sequence $\beta_j\rightarrow 0$, such that
    \[\cH^{n-1}(\Sigma_j\cap B_{\beta_j}(0))\ge (1+\ep)\cH^{n-1}((C_0\times \RR^{n-k-1})\cap B_{\beta_j}(0)).\]
    Since $\Sigma_j$ is of free boundary, and that the position vector field is tangential to $\partial C$, we the monotonicity formula holds for $\Sigma_j$ for balls centered at $0$. In particular, for each $\sigma\in [\beta_j,1]$, 
    \[\cH^{n-1}(\Sigma_j\cap B_{\sigma}(0))\ge (1+\ep)\cH^{n-1}((C_0\times \RR^{n-k-1})\cap B_{\sigma}(0)).\]
    However, by standard convergence theory for free boundary area minimizing hypersurfaces, $\Sigma_j$ converge (as integral currents) to $\Sigma$. By the upper semi-continuity of density and the fact that $\Sigma$ is area minimizing, we find that 
    \[\cH^{n-1}(\Sigma\cap B_\sigma(0))=\lim_{j\rightarrow \infty} \cH^{n-1}(\Sigma_j\cap B_\sigma(0)),\]
    for a dense set of $\sigma\in (0,1)$. On the other hand, by Theorem \ref{theo.regularity.of.free.boundary.minimizing.surface}, $\Sigma$ is $C^{1,\alpha}$ regular up to its corners, and $T_0\Sigma= C_0\times \RR^{n-k-1}$, contradiction.
\end{proof}

\begin{coro}\label{corollary.stable.bernstein}
Suppose $n\le 7$, $C=W_0\times [0,1)^{k-2}\times \RR^{n-k}$, and $\Sigma\subset C$ is a free boundary area minimizing hypersurface. Then $\Sigma$ is part of a hyperplane in $\RR^n$.
\end{coro}
\begin{proof}
    Without loss of generality assume $0\in \Sigma$, and $\nu(0)=(0,\cdots,0,1)$. For a real number $r>0$, consider the rescaled surface $\Sigma_r=r^{-1}\Sigma$. By Lemma \ref{lemma.curvature.estimates.local.epsilon}, there is a constant $\beta$ such that \eqref{equation.smallness.area.excess} holds for each $\Sigma_r$. By Theorem \ref{theo.regularity.allard.type}, there exist $\rho,\alpha>0$ such that $\Sigma_r\cap B_{\rho\beta}(0)$ is given as a graph $u_r$ with $[Du_r]_{0,\alpha,B_{\rho\beta}}<c_1$, for some constant $c_1$ independent of $r$. Since $u_r$ solves the minimal surface equation with Neumann boundary condition, by Appendix B, we have that $u_r\in C^{2,\alpha}(\overline{C})$. Thus, the Schauder estimate implies that
    \[|D^2 u|<c_2,\quad \text{in }B_{\rho\beta/2}(0).\]
    By scaling, this implies that 
    \[\sup_{B_{r\rho\beta/2}(0)}|A_{\Sigma}|<c_2r^{-1}.\]
    Taking $r\rightarrow \infty$, we have that $A_\Sigma=0$ everywhere.
\end{proof}

We now prove Theorem \ref{theo.curvature.estimate}.
\begin{proof}
    Suppose the statement is false, and there exists a sequence of free boundary area minimizing hypersurfaces $\Sigma_j$ such that $\lambda_j=\sup |A_{\Sigma_j}|\rightarrow \infty$. Note that each $\Sigma_j$ is a $C^{2,\alpha}$ regular hypersurface, thus $\sup |A_{\Sigma_j}|$ is achieved, with $|A_{\Sigma_j}|(x_j)=\lambda_j$. By possibly taking a further subsequence, we assume that $p_j\rightarrow p\in M$. Define the rescale $\eta_j$ on $M$ by letting $\eta_j(z):=\lambda_j(z-p_j)$. We then obtain a sequence of embedded area minimizing free boundary hypersurfaces
    \[(\Sigma'_j,\partial \Sigma'_j)=(\eta_j(\Sigma_j),\eta_j(\partial \Sigma_j))\subset (\eta_j (M), \partial (\eta_j M)).\]
    Moreover, $|A_{\Sigma'_j}|=1$, and $|A_{\Sigma'_j}|\le 1$ everywhere.
    
    The manifold $\eta_j (M)$, equipped with the induced Riemannian metric, converge, in the sense of $C^{2,\alpha}$ Cheeger-Gromov to a polyhedron $C\subset \RR^n$, where up to a rigid motion of $\RR^n$, $C= W_0\times [0,\infty)^{k-2}\times \RR^{n-k}$. Here $W_0$ is a two dimensional wedge whose opening angle is not larger than $\pi/2$. Since $\sup_{x\in \Sigma'_j} |A_{\Sigma'_j}|\le 1$ everywhere, locally $\Sigma'_j$ can be represented as the graph of a $C^2$ function satisfying an elliptic PDE with uniformly bounded coefficients and Neumann boundary condition. Thus, $\Sigma_j'$ converges locally smoothly as $C^{2,\alpha}$ graphs to a limit hypersurface $\Sigma_\infty$. Note that $\Sigma_\infty$ must coincide with the current limit of $\Sigma'_j$, which is a free boundary area minimizing hypersurface in $C$ (which exists by compactness of area minimizing currents). By Corollary \ref{corollary.stable.bernstein}, $\Sigma_\infty$ is planar. This contradicts the fact that $|A_{\Sigma'_j}|(0)=1$.
\end{proof}

\begin{rema}
    With a standard point picking trick, one can also obtain the following local curvature estimate: suppose $\Sigma$ is free boundary area minimizing hypersurface in $B_p(R)$, where $p\in M$, we have that 
    \[\sup_{x\in B_p(R/2)}|A_\Sigma|(x)\le C(p,R,M).\]
\end{rema}
\begin{rema}
    If one replaces the area minimizing assumption in Theorem \ref{theo.curvature.estimate} by stability, it is unclear whether a similar curvature estimates hold. For example, consider a domain $C=W_0\times \RR$ in $\RR^3$, where $W_0=\{(r,\theta): r\ge 0, \theta\in (0,\theta_0]\}$ is a non-obtuse wedge ($\theta_0\le \pi/2$). Then the plane $\Sigma :=\{\theta=\theta_1\}$ in $C$ is stable with respect to tangential variations along $\partial C$. However, an Allard type theorem fails for surfaces that are close to $\Sigma$, see \cite[Example 3.6]{EdelenLi2019regularity}. In particular, the usual blow up argument in the proof of curvature estimates fails in this situation.
    
\end{rema}

\section{Rigidity and splitting of minimal slicing}
In this section we are going to prove Theorem \ref{theo.dihedra.rigidity}. We have already seen that Theorem \ref{theo.dihedra.rigidity} follows directly from the Gauss-Bonnet theorem when $n=2$. Suppose Theorem \ref{theo.dihedra.rigidity} holds for all $(n-1)$ dimensional Riemannian polyhedron, over-prism of type $P$. We are going to prove that the same statement holds for dimension $n$. Consider the variational problem \eqref{problem.variation} and let $\Omega$ be the minimizer of it. Let $\Sigma=\Omega\cap \mathring{M}$. By Theorem \ref{theo.strong.maximum.principle}, we have either $\Sigma$ is disjoint from $F_T$ and $F_B$, in which case it is a free boundary area minimizing hypersurface, or $\Sigma$ entirely coincide with either $F_T$ (or $F_B$), in which case the face $F_T$ (or $F_B$) itself is a free boundary area minimizing hypersurface. In either case, by Theorem \ref{theo.regularity.of.free.boundary.minimizing.surface}, $\Sigma$ is $C^{2,\alpha}$ up to its corners. Moreover, $\Sigma$ is a Riemannian polyhedron, over-prism of type $P_1$, where $P_1=W_0\times [0,1]^{n-2}$. Note that since $\Sigma$ meets $\partial M$ orthogonally, its dihedral angle is everywhere equal to the constant dihedral angle given by $P_1$.

Since $\Sigma$ is free boundary and minimizing, the stability inequality implies that
\begin{equation}\label{equation.stability.inequality}
    Q(f,f)=\int_\Sigma |\nabla f|^2-(|A_\Sigma|^2+\Ric(\nu,\nu))f^2 dV-\int_{\partial \Sigma} \secondfund_{\partial M}(\nu,\nu)f^2 dS\ge 0,
\end{equation}
for any smooth function $f$ on $\Sigma$. Here $\secondfund_{\partial M}$ is the second fundamental form of $\partial M$ taken with respect to outward unit normal vector field.

By the Gauss equation, we have
\begin{equation}
    \Ric(\nu,\nu)+|A_\Sigma|^2=\frac{1}{2}(R_M-R_\Sigma+|A_\Sigma|^2),    
\end{equation}
where $R_M,R_\Sigma$ are the scalar curvature of $M,\Sigma$, respectively. The second variation form therefore becomes 
\begin{equation}
    Q(f,f)=\int_\Sigma |\nabla f|^2-\frac{1}{2}(R_M-R_\Sigma+|A_\Sigma|^2)f^2 dV-\int_{\partial \Sigma} \secondfund_{\partial M}(\nu,\nu)f^2 dS.
\end{equation}

Since $\Sigma$ is stable, the principal eigenvalue of the variational problem (see section 2 of \cite{MaximoNunesSmith2017freeboundary}, or \cite{Schoen2006minimal}) associated to the second variation form satisfies:
\begin{equation}
    \lambda_1=\inf_{\varphi\in L^2(\Sigma)}\frac{Q(\varphi,\varphi)}{\int_\Sigma \varphi^2 dV}\ge 0.
\end{equation}
The associated eigenfunction $\varphi\in W^{1,2}(\Sigma)$ satisfies, in the weak sense, the elliptic equation
\begin{equation}\label{equation.first.eigenvalue.stability.PDE}
    \begin{cases}
        \Delta_\Sigma \varphi+\frac{1}{2}(R_M-R_\Sigma+|A_\Sigma|^2)\varphi=-\lambda_1\varphi \quad&\text{ in }\Sigma,\\
        \pa{\varphi}{\eta}=\secondfund(\nu,\nu)\varphi\quad&\text{ on }\partial \Sigma.
    \end{cases}
\end{equation}

Here $\eta$ is the outward unit conormal vector field along $\partial \Sigma$. By assumption, the dihedral angles of adjacent faces of $\Sigma$ are all not larger than $\pi/2$. We then apply Proposition \ref{proposition.appendix.regularity} and conclude that the solution to \eqref{equation.first.eigenvalue.stability.PDE} is $C^{2,\alpha}$ to the corners. To do so, we verify that in the case that when two adjacent faces of $\Sigma$ meet orthogonally, the compatibility condition \eqref{equation.appendix.compatibility.condition} is satisfied. Notice that this condition is necessary for a solution to have H\"older continuous second derivatives.

\begin{lemm}\label{lemma.compatibility.condition}
    Suppose $F_i,F_j$ are two adjacent faces of $M$ meeting orthogonally, where $\Sigma\cap F_i,\Sigma\cap F_j$ are non-empty. Let $e_i,e_j$ be the outward normal vector field of $F_i,F_j$ in $M$. Then 
    \[\nabla_{e_j}(\secondfund_{F_i}(\nu,\nu))=\nabla_{e_i}(\secondfund_{F_j}(\nu,\nu)).\]
    holds along the intersection $F_i\cap F_j\cap \overline{\Sigma}$.
\end{lemm}

\begin{proof}
    Fix a point $p\in F_i\cap F_j\cap \overline{\Sigma}$. Extend $e_i(p),e_j(p),\nu(p)$ to be a geodesic normal coordinate frame in a neighborhood of $p$, such that at $p$, all covariant derivatives are zero. We apply the Codazzi equation on the hypersurface $F_i\subset M$, and obtain
    \[\nabla \secondfund_{F_i}(e_j,\nu,\nu)=\nabla \secondfund_{F_i}(\nu,e_j,\nu).\]
    Since the local coordinates is normal,
    \[\nabla_{e_j}(\secondfund_{F_i}(\nu,\nu))=\nabla \secondfund_{F_i}(e_j,\nu,\nu)=\nabla\secondfund_{F_i}(\nu,e_j,\nu)=\nabla_\nu\bangle{\nabla_{e_i}e_j,\nu}\]
    holds at $p$. Similarly, $\nabla_{e_i}(\secondfund_{F_j}(\nu,\nu))=\nabla_\nu\bangle{\nabla_{e_j}e_i,\nu}$ at $p$. Notice that $\nabla_{e_i}e_j=\nabla_{e_j}e_i$, hence the conclusion of lemma holds.
\end{proof}

Given Lemma \ref{lemma.compatibility.condition}, the solution $\varphi$ to \eqref{equation.first.eigenvalue.stability.PDE} is in $C^{2,\alpha}(\overline{\Sigma})$, by virtue of the Proposition \ref{proposition.appendix.regularity} in Appendix B.

Let $g_1$ be the induced metric on $\Sigma$. Using the induction hypothesis, we prove the following property of $\Sigma$.

\begin{prop}\label{proposition.sigma.infinitesimally.rigid}
    The hypersurface $\Sigma$, equipped with the induced metric $g_1$, is isometric to an Euclidean prism. Moreover, $\Sigma\subset M$ is infinitesimally rigid: we have that
    \[\Ric_M(\nu,\nu)=0,\quad |A_\Sigma|=0\text{ on }\Sigma,\]
    and $\secondfund_{\partial M}(\nu,\nu)=0$ on $\partial \Sigma$.
\end{prop}

\begin{proof}
    Let $\varphi$ be a positive function on $\Sigma$, under the conformal change of metric 
    \[g_2=\varphi^{\frac{2}{n-2}}g_1,\]
    the scalar curvature changes by
    \begin{equation}\label{equation.conformal.change.scalar.curvature}
        R(g_2)=\varphi^{-\frac{n}{n-2}}\left(-2\Delta\varphi+R(g_1)\varphi+\frac{n-1}{n-2}\frac{|\nabla\varphi|^2}{\varphi}\right);
    \end{equation}
    Where $\Delta$ and $\nabla$ are taken with respect to the metric $g_1$. The mean curvature of $\partial \Sigma\subset \Sigma$ with respect to the outward unit conormal vector field changes via
    \begin{equation}\label{equation.conformal.change.mean.curvature}
        H_{\partial \Sigma}(g_2)=\varphi^{-\frac{1}{n-2}}\left(H_{\partial \Sigma}(g_1)+\frac{1}{\varphi}\pa{\varphi}{\eta}\right).
    \end{equation}
    
    Now we choose $\varphi>0$ to be the solution to \eqref{equation.first.eigenvalue.stability.PDE}. Then the scalar curvature satisfies
    \begin{equation}\label{equation.conformal.change.scalar.curvature.simplified}
        R(g_2)=\varphi^{-\frac{n}{n-2}}\left((R_M+|A_\Sigma|^2+\lambda_1)\varphi+\frac{n-1}{n-2}\frac{|\nabla\varphi|^2}{\varphi}\right)\ge 0.
    \end{equation}
    
    The boundary mean curvature becomes
    \[H_{\partial \Sigma}(g_2)=\varphi^{-\frac{1}{n-1}}(H_{\partial \Sigma}(g_1)+\secondfund(\nu,\nu)).\]
    Let $\{e_j\}_{j=1}^{n-2}$ be an orthonormal frame in an open set of $\partial \Sigma$. Since $\Sigma$ meets $\partial M$ orthogonally, the conormal vector $\eta$ of $\partial \Sigma$ in $\Sigma$ is the same as the conormal vector of $\partial M$ in $M$. Therefore $e_j$, $j=1,\cdots, n-2$, $\nu$ and $\eta$ forms an orthogonal basis in an open neighborhood of $\partial M$. Therefore, we verify that
    \begin{equation}
            H_{\partial \Sigma}(g_1)+\secondfund(\nu,\nu)=\sum_{j=1}^{n-2}\bangle{\nabla_{e_j}e_j,\eta}+\bangle{\nabla_\nu \nu,\eta}=\overline{H}_{\partial M}.
    \end{equation}
    Here $\overline{H}_{\partial M}$ represents the mean curvature of $\partial M$ in $M$, with respect to the outward unit normal vector field $\nu$. We therefore conclude that
    \begin{equation}\label{equation.conformal.change.mean.curvature.simplified}
        H_{\partial \Sigma}(g_2)=\varphi^{-\frac{1}{n-2}}\overline{H}_{\partial M}\ge 0.
    \end{equation}
    
    Observe that the dihedral angles of $\Sigma$ are equal to the dihedral angles of $M$, since $\Sigma$ meets $\partial M$ orthogonally. Moreover, the conformal deformation does not change dihedral angles. We therefore conclude, by \eqref{equation.conformal.change.scalar.curvature.simplified} and \eqref{equation.conformal.change.mean.curvature.simplified}, that $(\Sigma,g_2)$ is an over-prism manifold of dimension $(n-1)$ with nonnegative scalar curvature, weakly mean convex faces, and everywhere non-obtuse dihedral angles. By induction, $(\Sigma,g_2)$ is isometric to an Euclidean rectangular solid. Tracking equalities in \eqref{equation.conformal.change.scalar.curvature.simplified} and \eqref{equation.conformal.change.mean.curvature.simplified}, we conclude that, with respect to metric $g_1$:
    \[R_M=0,\quad A_\Sigma=0,\quad \lambda_1=0,\quad \nabla\varphi=0 \quad \text{ on }\Sigma,\]
    and $\overline{H}_{\partial M}=0$ on $\partial \Sigma$.
    
    In particular, $\varphi$ is a constant function. This implies that $(\Sigma,g_1)$ is also isometric to a flat Euclidean rectangular solid. The fact that $(\Sigma,g_1)$ is infinitesimally rigid then follows.
\end{proof}

We then proceed to the proof of Theorem \ref{theo.dihedra.rigidity}. The basic idea is to conformally deform the metric locally near the infinitesimally rigid hypersurface $\Sigma$, and solve the variational problem \eqref{problem.variation} in a new metric. As a result, we are able to conclude that each side of $\Sigma$ contains a dense collection of flat Euclidean rectangular solids. This important technique first appeared in \cite{CarlottoChodoshEichmair2016effective}, and were also used in \cite{ChodoshEichmairMoraru2018splitting}. Though these papers are written only for $3$ dimensional manifolds, we observe that the idea within can be extended in higher dimensions, as long as the minimizer to the relevant variational problem is regular.

\begin{proof}[Proof of Theorem \ref{theo.dihedra.rigidity}]
    Let $(M^n,g)$ be an over-prism polyhedron as in the assumption of Theorem \ref{theo.dihedra.rigidity}, $\Sigma^{n-1}\subset M$ be a free boundary area minimizing hypersurface. Suppose $M_+\subset M$ is a region separated by $\Sigma$ containing $F_T$. By the strong maximum principle (Theorem \ref{theo.strong.maximum.principle}), either $\Sigma$ is disjoint with $F_T$ and $F_B$, or $\Sigma$ completely coincides with $F_T$ or $F_B$. In the latter case, $\Sigma=F_T$ (or $F_B$), and itself is an area minimizing hypersurface. 
    
    Without loss of generality assume that $\Sigma\ne F_T$ so $M_+\ne \emptyset$. Fix $r_0>0$ small to be determined later. Fix a point $p\in M_+$, such that $\dist_g(p,\Sigma)\in (1.5r_0,2.5r_0)$ and $\dist_g(p,\partial M)>4r_0$. Denote by $r(x)$ the distance function to $p$ with respect to the metric $g$. We take $r_0$ small enough, such that the function $\pa{r}{\nu}>0$ in $\Sigma\cap B_{3r_0}(p)$.  
    
    \textbf{Claim:} There exists an $\ep>0$ and a family of Riemannian metrics $\{g(t)\}_{t\in (0,\ep)}$ on $M$ in the conformal class of $g$, with the following properties:
    \begin{enumerate}
        \item $g(t)\rightarrow g$ smoothly as $t\rightarrow 0$;
        \item $g(t)=g$ on $M\setminus B_{3r_0}(p)$;
        \item $g(t)\le g$ as metrics on $M$, with strict inequality on $ B_{3r_0}(p)\setminus B_{r_0}(p)$.
        \item $R(g(t))>0$ on $ B_{3r_0}(p)\setminus B_{r_0}(p)$.
        \item Surface $\Sigma$ is weakly mean convex and strictly mean convex at one interior point with respect to the metric $g(t)$, and the unit normal vector field pointing outward from $M_+$.
    \end{enumerate}
    
    The construction of these conformal metrics follows from an extension of the Appendix J of \cite{CarlottoChodoshEichmair2016effective}, which we describe here for the sake of completeness.
    
    Let $f\in C^\infty(\RR)$ be a non-positive function supported in $[0,3]$, such that in $(1,3)$, 
    \[f(s)=-\exp \left(\frac{16n+8}{s}\right).\]
    This is to make sure that $f$ satisfies the inequalities $f'(s)>0$, $(4n-1)f'(s)+sf''(s)<0$ for $s\in (1,3)$.
    
    Let $r_0<\inj (M,g)$ be small enough to guarantee that in any geodesic ball $B_{3r_0}(q)$, the inequality $\Delta_g (\dist^2_g(\cdot,q))<8n$ is satisfied. On $M$, define $v(x)=r_0^2 f(r(x)/r_0)$. In $B_{3r_0}(p)\setminus B_{r_0}(p)$, we then check that
    \begin{multline}
        \Delta_g v=r_0 f'\left(\frac{r}{r_0}\right)\Delta_g r+f''\left(\frac{r}{r_0}\right)\\
            <(4n-1)\frac{r_0}{r} f'\left(\frac{r}{r_0}\right)+f''\left(\frac{r}{r_0}\right)<0.
    \end{multline}
    
    The metric $g(t)$ is then defined as $g(t)=(1+tv)^{\frac{4}{n-2}}g$. We verify that conditions (1)-(5) are satisfied by $\{g(t)\}$. Conditions (1)-(3) are straightforward. To see (4), we note that
    \[R(g(t))=(1+tv)^{-\frac{n+2}{n-2}}\left(\frac{-4(n-1)t}{n-2}\Delta_g v +R(g) (1+tv) \right)>0.\]
    To see (5), we notice that along $\Sigma$, 
    \[H_{\Sigma,g(t)}=(1+tv)^{-\frac{n}{n-2}}\left(H_{\Sigma,g}+\frac{2(n-1)t}{n-2}\pa{v}{\nu}\right)\ge 0.\]
    The claim is proved.
    
    To proceed, observe that $(M_+,g(t))$ is an overcubic manifold that satisfies all the conditions satisfied by $(M,g)$. Now $M_+$ has top face $F_T$ and bottom face $\Sigma$. We consider the variational problem in $M_+$ analogous to \eqref{problem.variation}:
    \begin{multline*}\label{problem.variation.2}
        I=\inf\{\cH^{n-1}(\partial \Omega\cap \mathring{M_+})_{g(t)}:\\
         \text{ is an open set of finite perimeter containing }F_T, \Omega\cap \Sigma=\emptyset.\}
    \end{multline*}
    By the strong maximum principle (Theorem \ref{theo.strong.maximum.principle}) and the regularity theory (Theorem \ref{theo.regularity.of.free.boundary.minimizing.surface}), $I$ is achieved by an open set $\Omega_t$. Consider $\Sigma_t=\Omega_t\cap \mathring{M}$. Then $\Sigma_t$ is $C^{2,\alpha}$ to its corners. Notice that $\Sigma_t$ is disjoint from $\Sigma$ by the convexity condition (5) of the metric $g(t)$. Also $\Sigma_t$ must intersect $B_{3r_0}(p)$, otherwise we would have:
    \[|\Sigma|_{g}\le |\Sigma_t|_{g}= |\Sigma_t|_{g_t}\le |\Sigma|_{g_t}<|\Sigma|_g,\]
    contradiction (where the last inequality comes from condition (4)). Also, Proposition \ref{proposition.sigma.infinitesimally.rigid} implies that $R(g(t))\equiv 0$ on $\Sigma_t$. Since $R(g(t))>0$ in $B_{3r_0}(p)\setminus B_{r_0}(p)$, we conclude that $\Sigma_t\cap B_{r_0}(p)\ne \emptyset$. 
    
    As $t\rightarrow 0$, consider the family of hypersurfaces $\{\Sigma_t\}_{t\in (0,\ep)}$. By the curvature estimate (Theorem \ref{theo.curvature.estimate}), $\Sigma_t$ subsequentially converges in $C^{2,\alpha}$ to a free boundary area minimizing surface $\Sigma'\subset M_+$. Moreover, $\Sigma'\cap B_{r_0}(p)\ne \emptyset$. In particular, $\Sigma'$ and $\Sigma$ are disjoint.
    
    This argument, carried out on each of the free boundary minimizing hypersurface, with varying choices of the small radius $r_0$ and point $p$, implies that the whole region $M_+$ (and $M\setminus M_+$ likewise) contains a collection of free boundary area minimizing hypersurfaces $\{\Sigma^{\rho}\}_{\rho\in \cA}$, where $\cup_{\rho\in \cA} \Sigma^{\rho}$ is dense in $M_+$. By Proposition \ref{proposition.sigma.infinitesimally.rigid}, each $\Sigma^{\rho}$ is isometric to an Euclidean rectangular solid, and is also infinitesimally flat in $M$. We also observe that $\partial \Sigma^\rho$ is also dense in the boundary. Precisely, for any point $q\in \partial M^+$ and any $\ep>0$, there is a free boundary area minimizing surface $\Sigma^\rho$ that intersects $B_\ep(q)\cap M_+$, $\rho\in \cA$. By the curvature estimate (Theorem \ref{theo.curvature.estimate}), $\Sigma^\rho$ must also intersects $\partial M^+$ in some $B_{C\ep}(q)$, for some uniform constant $C$ (independent of the choice of $q$ and $\ep$).
    
    Pick a hypersurface $\Sigma^\rho$. For notational simplicity, assume without loss of generality that $\rho=0$ and $\Sigma^\rho=\Sigma$. Extend the unit normal vector field $\nu$ of $\Sigma$ to a smooth vector field $X$, such that $X$ is tangential on $\partial M$. Let $\phi_t$ be the one-parameter family of diffeomorphisms generated by $\phi_t$. For $\rho$ sufficiently small, $\Sigma^\rho$ can be represented as a graph
    \[\{\phi_{u^\rho(x)}(x):x\in \Sigma\}.\]
    Fix a point $x_0\in \Sigma$. Arguing as (2) in \cite{liu2013three-manifolds}, the functions $\{u^\rho(x)/u(x_0)\}$ converges in $C^{2,\alpha}(\overline{\Sigma})$, as $\rho\rightarrow 0$, to a function $u$ that satisfies
    \[(\nabla_\Sigma^2 u)(X,Y)+Rm_M (\nu,X,Y,\nu)=0,\quad \text{for all tangential vectors } X,Y.\]
    Taking trace, we have $\Delta_\Sigma u=0$. Also, since each $\Sigma^\rho$ meets $\partial M$ orthogonally, we have $\pa{u}{\eta}=0$ on $\partial \Sigma$. We therefore conclude that $u=0$ on $\Sigma$, and hence $Rm_M(\nu,X,Y,\nu)=0$. Using the Gauss-Codazzi equations on $\Sigma$, and the fact that $|A_\Sigma|=0$, we then obtain that $Rm_M=0$ on $\overline{\Sigma}$. By density of $\{\Sigma^\rho\}$, we conclude that $Rm_M=0$ in $M_+$. By Proposition \ref{proposition.sigma.infinitesimally.rigid}, $\secondfund=0$ on each face of $\partial\Sigma^\rho$. Since $\Sigma^\rho$ is dense in $\partial M_+$, we conclude that each face of $\partial M_+$ is totally geodesic. This implies that $M_+$ is isometric to an Euclidean prism. Combined with a similar argument for $M\setminus M_+$, we conclude that $(M,g)$ is isometric to an Euclidean prism.
\end{proof}

\begin{rema}
    We remark that the geometric argument in section 4 is very robust, and can potentially be used for Conjecture \ref{conj:dihedral.rigidity} for different polytopes, as long as the corresponding variational problem produces a $C^{2,\alpha}$ hypersurface. For instance, if the minimizer to \eqref{problem.variation.capillary} is $C^{2,\alpha}$ to its corners, then one can establish Conjecture \ref{conj:dihedral.rigidity} for $n$-dimensional simplices.
\end{rema}

\section{Application to the positive mass theorem}
In this section we observe that Conjecture \ref{conj:dihedral.rigidity} implies the positive mass theorem in a rather straightforward manner. In fact, it would be enough to assume Conjecture \ref{conj:dihedral.rigidity} holds for a \textit{single} Euclidean polyhedron $P$:

\begin{conj}\label{conj.dihedral.rigidity.for.P}
Suppose Conjecture \ref{conj:dihedral.rigidity} holds for $P$. That is, for any Riemannian polyhedron $(M,g)$ admitting a degree one map onto $P$, conditions (1)-(3) in Conjecture \ref{conj:dihedral.rigidity} imply that $(M,g)$ is flat. 
\end{conj}

We will see that this is enough to imply the positive mass theorem for an $n$-dimensional asymptotically flat manifold. \footnote{In fact, it suffices to assume the dihedral angles of $M$ is everywhere equal to the corresponding angles of $P$, as we did in Theorem \ref{theo.dihedra.rigidity}.}

\begin{theo}
    Suppose Conjecture \ref{conj.dihedral.rigidity.for.P} holds for a certain Euclidean polyhedron $P$. Then the positive mass theorem holds: Suppose $(X^n,g)$ is asymptotically flat manifold, then the ADM mass on each end of $X$ is nonnegative.
\end{theo}

\begin{proof}
    \footnote{The proof we give here assumes that $X$ has one end. The general case can be handled similarly by cutting the other ends along strictly mean convex hypersurfaces of $X$.}Suppose $X\setminus K$ is diffeomorphic to $\RR^n\setminus B_r(0)$ with $K$ compact, and the ADM mass is negative. By the work of \cite{SchoenYau1979positivity} and an observation of Lohkamp, there exists a non-flat metric $\tilde{g}$ on $X$ with $R(\tilde{g})\ge 0$, and $\tilde{g}$ is isometric to the Euclidean metric on $X\setminus K$. Since $P$ is an Euclidean polyhedron, one can take a sufficiently large rescale of $P$, which we also denote by $P$, such that $\partial P$ lies entirely inside $\RR^n\setminus B_R(0)$. Now $\partial P$, isometrically embedded as a boundary in $X\setminus K$, bounds a Riemannian polyhedron $M$. Observe that $M$ admits a degree one map onto $P$ by simply taking $K$ to $\{0\}$ and $M\setminus K$ to $P\setminus \{0\}$. However, $(M,\tilde{g})$ satisfies all 3 conditions in Conjecture \ref{conj.dihedral.rigidity.for.P}, thus $\tilde{g}$ should be flat, contradiction.
\end{proof}

\begin{rema}
	Our approach in the proof of ridigity statement of Theorem \ref{theo.dihedra.rigidity} implies the following alternative proof of the rigidity statement of the positive energy theorem: after the Lohkamp reduction, place a large cube $C$ in $X$ such that $g$ is flat in $X\setminus C$. Then identify opposite sides of $C$ and obtain a metric $\tilde g$ on $T^n \# K$, where $K$ is a closed manifold. Find a two-sided area minimizing hypersurface $\Sigma$ in $T^n\# K$ which is the Poincar\'e duality of a generator of $H^1(T^n)$. The Schoen-Yau dimension descent argument implies that $\Sigma$ is intrinsically flat and $\Ric(\nu,\nu)=0$ on $\Sigma$. Now an analogous metric deformation argument as in the proof of Theorem \ref{theo.dihedra.rigidity} implies that $T^n\# K$ contains a dense collection of intrinsically flat, infinitesimally rigid flat $T^{n-1}$, and thus $\tilde g$ itself is flat.
\end{rema}

\appendix

\section{Bending construction of hypersurfaces}\label{appendix.section.bending.construction}

In this section we review a bending construction due to Gromov. For more details, we refer the readers to of \cite[Section 11.3]{Gromov2018Metric}. Our objective is the following proposition, which reduces the Conjecture \ref{conj:dihedral.rigidity} to the statement of Theorem \ref{theo.dihedra.rigidity}, by fixing the dihedral angles between faces of a polyhedron.

\begin{prop}\label{prop.bending.construction}
Assume $(M^n,g)$ is a Riemannian polyhedron with a $C^\infty$ metric $g$, such that
\begin{enumerate}
    \item Each face of $M$ is weakly mean convex;
    \item There exists a smooth function $\alpha$ defined on the $(n-2)$-skeleton $E$ of $M$, which is constant along each edge $F_i\cap F_j$, such that for any point $q\in E$, the dihedral angle $q$ is less than or equal to $\alpha(q)$.
\end{enumerate}
Then there exists a Riemannian polyhedron $M'\subset M$, such that with the induced metric, $M'$ satisfies that
\begin{enumerate}
    \item Each face of $M'$ is weakly mean convex;
    \item The dihedral angle at every $q\in E$ is equal to $\alpha(q)$.
\end{enumerate}
\end{prop}

Gromov's idea of proof of Proposition \ref{prop.bending.construction} is by induction, which we very briefly describe as follows. Take one face $F_0 \subset M$. We fix a vector field $X$ in $M$, such that $X$ is transversal to $\Sigma_0$ and points inward $M$, but tangential to all the other faces of $M$. Let $\phi=\phi(x,t)$ be the flow generated by the vector field $X$. Let $f$ be a smooth function defined on $\Sigma_0$. Define
\[F_f=\{\phi(p,f(p):x\in F_0)\}.\]
By the tubular neighborhood theorem, $\Sigma_f$ is isotopic to $\Sigma_0$ when $\|f\|_{C^0}$ is small enough. Hence by replacing the face $F_0$ by $F_f$ and keeping all the other faces the same, we may obtain a new polyhedron $M'\subset M$. It remains to check the conditions on mean curvature and on dihedral angle. Let $H_f$ and $\theta_f$ be the mean curvature (with respect to outward unit normal) and the dihedral angle of the surface $F_f$. The linearized operator of $H_f$ and $\theta_f$ is then given by 
\begin{equation}
        \begin{split}
            &\partial_t|_{t=0} H_{(tf)} = \Delta f + (\Ric_M(\nu,\nu)+|A|^2)f,\\
            &\partial_t|_{t=0} \cos\theta_{(tf)} = -\pa{f}{\eta} + \bangle{\nu,\nabla_{\nu} X} f.
        \end{split}
\end{equation}

Denote $\Lambda>0$ be an upper bound of $\Ric_M(\nu,\nu)+|A|^2$ and $\bangle{\nu,\nabla_{\nu}X}$. Then one may choose a function $f$ such that $|f|_{C^0}$ is  sufficiently small, but $\Delta f\ge 2\Lambda |f|$, $\pa{f}{\eta}\le -2\Lambda |f|$. (See, for instance, the function $f$ in the proof of Theorem \ref{theo.dihedra.rigidity}.) By the implicit function theorem, we then may strictly increase the mean curvature and the dihedral angle the same time. Notice that during this process, we do not change the dihedral angle between other faces.

Once we have fixed the dihedral angle between $F_0$ and other faces, we may repeat this process and inductively increase the dihedral angle between all pairs of faces. 

Note that when $M$ is over-prism of type $P$ with $\Phi: M\to P$ the degree one map, and $p\in M$ is on an edge such that the dihedral angle at $p$ is strictly less than that of $\Phi(p)$, the bending construction will produce $M'$ whose boundary mean curvature in a neighborhood of $p$ is strictly mean convex. In particular, Theorem \ref{theo.dihedra.rigidity} and the bending construction implies the strong form of dihedral rigidity, where we only assume dihedral angles everywhere less than or equal to that of $P$.

\subsection{Application to dihedral rigidity in dimension three}
As an application of Proposition \ref{prop.bending.construction}, we may extend the main results of \cite{Li2017polyhedron} to more general polyhedra of cone or prism types. In Theorem 1.4 and 1.5 of \cite{Li2017polyhedron}, for the purpose of regularity of the capillary functional minimizer, an extra condition on the dihedral angles was needed. Let us briefly recall these results. Given an Euclidean flat cone or prism $P$ and a Riemannian polyhedron $(M^3,g)$ diffeomorphic to $P$, let $F_j$, $j=1,\cdots,k$ be the side faces of $M$, $F_j'$ be the side faces of $P$. Let $B'$ be the base face of $P$. Denote the (constant) dihedral angle between $F_j'$ and $B'$ by $\gamma_j$. Then in Theorem 1.4 and 1.5 of \cite{Li2017polyhedron}, it was assumed that
\begin{align}
   & |\pi-(\gamma_j+\gamma_{j+1})|<\measuredangle(F_j,F_{j+1}), \label{assumption.angle.between.side.faces}\\
    &\gamma_j\le \frac{\pi}{2},\quad j=1,\cdots, k. \label{assumption.angle.non.obtuse}
\end{align}

\begin{theo}[Theorem 1.4 of \cite{Li2017polyhedron}]
Assume $P$ is an Euclidean simplex or prism, $(M^3,g)$ is a Riemannian polyhedron diffeomorphic to $P$, such that \eqref{assumption.angle.between.side.faces} and \eqref{assumption.angle.non.obtuse} holds. Then $(M,g)$ cannot simultaneously satisfy that $R(g)\ge 0$ in the interior of $M$, each face of $M$ is mean convex, and that the dihedral angles of $M$ is everywhere less than the corresponding dihedral angle of $P$.
\end{theo}

Notice that Proposition \ref{prop.bending.construction} enables us to increase the dihedral of a Riemannian polyhedron, and at the same time do not decrease face mean curvature and the interior scalar curvature. Therefore the lower bound in condition \eqref{assumption.angle.between.side.faces} may, without loss of generality, be dropped.

\begin{theo}
Let $P$ be a simplex or a prism in $\RR^3$, $g_0$ is the Euclidean metric on $P$ satisfying \eqref{assumption.angle.non.obtuse}. Suppose $g$ a Riemannian metric on $P$, such that $R(g)\ge 0$ in the interior of $M$, each face of $M$ is weakly mean convex, and the dihedral angles of $(P,g)$ is everywhere not larger than the corresponding dihedral angle of $(P,g_0)$. Then up to a scaling, $(P,g)$ is isometric to $(P,g_0)$.
\end{theo}

\section{Elliptic regularity in Riemannian polyhedron}
We prove some regularity results for second order elliptic equations in a Riemannian polyhedron. For the purpose of this paper, we only consider the case of a Riemannian prism, as in Definition \ref{definition.riemann.prism}. Let $(M^n,g)$ be a Riemannian prism with (closed) faces $\{F_i\}$. We assume that each face $F_i$ is a $C^{2,\alpha}$ hypersurface, and each pair of adjacent faces $F_i,F_j$ intersects transversely at constant angle. The Riemannian metric $g$ is also assumed to be $C^{2,\alpha}$ in $\overline{M}$. Consider an elliptic equation
\begin{equation}\label{equation.for.regularity.appendix}
    \begin{cases}
        Lu=0\quad &\text{in }M,\\
        \pa{u}{\nu}=qu\quad &\text{on }\partial M.
    \end{cases}
\end{equation}
Here, we assume that $L$ is a linear elliptic operator, such that the second order part is Laplacian. (Later we will apply this appendix to study regularity of minimal surfaces in $M$, and regularity of solutions to \eqref{equation.first.eigenvalue.stability.PDE}.) We also assume that $q$ is a function, $C^{1,\alpha}$ when restricted on each face of $M$, and satisfies the compatibility condition 
\begin{equation}\label{equation.appendix.compatibility.condition}
    \paop{\nu_j}(q|_{F_i})=\paop{\nu_i}(q|_{F_j})
\end{equation}
along the intersection $F_i\cap F_j$, for each pair of \textit{orthogonal} adjacent faces $F_i,F_j$.

Since $(M,g)$ is a Lipschitz domain, by classical theory in elliptic PDE, we know that any weak solution to \eqref{equation.for.regularity.appendix} is in $C^{0,\alpha}(\overline{M})$, and is locally $C^{2,\alpha}$ in the interior of $M$ and on the smooth part of $\partial M$ (namely, on the $(n-1)$ dimensional strata of $\partial M$). Hence it suffices to consider \eqref{equation.for.regularity.appendix} on the corners of $M$ (namely, on $d$ dimensional strata of $\partial M$ with $0\le d\le n-2$). Since the regularity theory is local in nature, we can fix a point $p$ at the corner. A fundamental observation to utilize the dihedral angle condition is to write the equation in the coordinate system constructed in Lemma \ref{lemma.local.foliation.2}. 

Precisely, suppose a neighborhood of $p$ in $M$ is diffeomorphic to the domain $\Omega=\{(x_1,\cdots,x_n):x_j\ge 0, j=1,\cdots,k\}$,
for some integer $k\in [2,n]$. (In the case where $k=1$, $p$ is on the smooth part of the boundary, and the solution is automatically $C^{2,\alpha}$.) Then $p$ lies on the intersection of $k$ faces of $M$, say $F_1,\cdots,F_k$. Note that $F_i$ meets $F_j$ along an interior constant dihedral angle $\gamma_{ij}$. Moreover, by definition \ref{definition.euclid.prisms}, for all but at most one pair of $(i,j)$, $i\le j$, $\gamma_{ij}=\pi/2$. Denote $\theta_0=\gamma_{12}$. Then $\theta_0\in (0,\pi/2]$. Take the coordinate system $(x_1,\cdots,x_n)$ constructed in Lemma \ref{lemma.local.foliation.2}. Then $\partial_i$ is the outward unit normal vector for $F_i$. We further make a change of coordinates
\[x_2'=(\cot\theta_0) x_1+\frac{1}{\sin\theta_0} x_2.\]
By slight abuse of notation, we will write $x_2$ for $x_2'$. As a consequence, $g(\partial_i,\partial_j)=0$ along each $F_i$, whenever $i\ne j$ and $1\le i,j\le k$. To illustrate the advantage of this coordinate system, we note that an Euclidean prism will be given by $W_2\times [0,\infty)^{k-2}\times \RR^{n-k}$, where $W_2$ is a wedge region on the $x_1,x_2$ plane. The equation \eqref{equation.for.regularity.appendix} can be written in the following form:

\begin{equation}\label{equation.appendix.euclidean.prism}
    \begin{cases}
        a_{ij}u_{ij}+b_i u_i+cu = f \quad &\text{in }\Omega\cap B_1,\\
        \pa{u}{\nu}=0 \quad &\text{on }\partial \Omega \cap B_1,\\
        u=g \quad &\text{on }\Omega\cap \partial B_1.
    \end{cases}
\end{equation}
where the coefficients satisfies
\begin{equation}\label{equation.appendix.euclidean.assumptions}
    \begin{aligned}
        &\Lambda^{-1}|\xi|^2 \le a_{ij}\xi^i\xi^j\le \Lambda|\xi|^2\\
        &a_{ij},b_i,c \in C^{0,\alpha}(\overline{\Omega\cap B_1}) \quad \text{with}\quad |a_{ij}|_{0,\alpha}, |b_i|_{0,\alpha},|c|_{0,\alpha}\le \Lambda.\\
        & a_{ij}=0 \text{ on } \{x_i=0\}, \text{whenever }j\ne i, 1\le i,j\le k.
    \end{aligned}
\end{equation}

We note that the equations in our geometric argument satisfies this property:

\begin{rema}
    In local coordinates, the Laplacian operator takes the form 
    \[\Delta_g u= \frac{1}{\sqrt{\det g}}\partial_i\left(\sqrt{\det g}g^{ij} \partial_j u\right)=g^{ij} u_{ij}+\text{lower order terms}.\]
    Since $g_{ij=0}$ on $F_i$ for $1\le i,j\le k$, $i \ne j$, \eqref{equation.appendix.euclidean.assumptions} is satisfied.
\end{rema}

\begin{rema}
    We will see that the minimal surface equation also satisfies the assumptions \eqref{equation.appendix.euclidean.assumptions}. By Theorem \ref{theo.regularity.of.free.boundary.minimizing.surface}, a free boundary area minimizing hypersurface in $M$ is locally a $C^{1,\alpha}$ graph over its tangent plane. The minimal surface equation takes the form
    \[F(g_{ij},u,u_i)u_{ij}+\text{lower order terms}=0\]
    with Neumann boundary conditions. Here $F$ is a smooth function depending analytically on the metric $g_{ij}$, the graphical function $u$ and its first derivatives. In particular, if $u_i=0$ along $x_i=0$, and $j\ne i$, $F_{ij}=0$.
\end{rema}

\begin{prop}\label{proposition.appendix.regularity}
    Suppose $u\in C^{1,\alpha}(\overline{\Omega\cap B_1})\cap W^{2,2}(\Omega\cap B_1)$ is a weak solution to \eqref{equation.appendix.euclidean.prism}, where the coefficient satisfies \eqref{equation.appendix.euclidean.assumptions}, and $f\in C^{0,\alpha}(\overline{\Omega\cap B_1})$, $g\in C^0(\overline{\Omega})$. Then $u\in C^{2,\alpha}(\overline{\Omega\cap B_{1/2}})$.
\end{prop}

\begin{proof}
    By the standard regularity theory, one has $C^{2,\alpha}$ regularity for $u$ in the interior of $\Omega\cap B_1$ and along the smooth part of $\partial\Omega\cap B_1$. We first rewrite the equation as
    \[a_{ij}u_{ij}=f-b_i u_i-cu,\]
    with the same boundary conditions.
    
    Denote $W=W_2\times \RR^{n-2}$. Consider the even extension $\bar{u}$ of $u$, defined in $W\cap B_1$, by letting 
    \[\bar{u}(x_1,\cdots,x_n)=u(x_1,x_2,|x_3|,\cdots,|x_k|, x_{k+1},\cdots, x_n).\] 
    Since $u\in C^{1,\alpha}(\overline{\Omega\cap B_1})$ and $\pa{u}{\nu}=0$ on $\partial W\cap B_1$, $\bar{u}\in C^{1,\alpha}(\overline{W}\cap B_1)$. We prove that $\bar{u}$ is the weak solution to an elliptic equation with $C^{0,\alpha}$ coefficients in $W\cap B_1$. Fix a point $x=(x_1,\cdots,x_n)$. For each pair of $(i,j)$, denote an integer $\tau_{ij}\in \{0,1,2\}$ as follows: if $i=j$, then $\tau_{ij}=0$; otherwise, $\tau_{ij}$ is the number of the elements $k\in \{i,j\}$ where $x_k<0$. Define 
    \[\bar{a}_{ij}(x_1,\cdots,x_n)=(-1)^{\tau_{ij}(x_1,\cdots,x_n)} a_{ij}(x_1,x_2,|x_3|,\cdots,|x_k|,x_{k+1},\cdots,x_n).\]
    Notice that $\bar{a}_{ij}$ is a $C^{0,\alpha}$ function in $B_1$. In fact, across each face $x_k=0$, either $\bar{a}_{ij}$ is an even extension of $a_{ij}$, or $\bar{a}_{ij}$ is an odd extension with $a_{ij}=0$ along $x_k$. 
    
    Similarly, for each $i$, define $\tau_i\in \{0,1\}$ as follows: if $x_i\ge 0$ then $\tau_i=0$; otherwise let $\tau_i=1$. Then define
    \[\bar{b}_i(x_1,\cdots,x_n)=(-1)^{\tau_i(x_1,\cdots,x_n)}b(x_1,x_2,|x_3|,\cdots,|x_k|,x_{k+1},\cdots,x_n).\]
    We observe that the function $\bar{b}_i\bar{u}_i$ is a $C^{0,\alpha}$ function in $B_1$. Indeed, along each $x_k=0$, either $i\ne k$ and $\bar{b}_i \bar{u}_i$ is an even extension, or $i=k$ and $u_i=0$ along $x_i=0$.
    
    Then we define $\bar{c},\bar{f},\bar{g}$ as the even extension of $c,f,g$. Thus $\bar{u}$ is a weak solution to 
    \[\bar{a}_{ij} \bar{u}_{ij}=\bar{f}-\bar{b_i}\bar{u}_i-\bar{c}\bar{u} \text{ in }W\cap B_1, \quad \bar{u}=\bar{g}\text{ on }W\cap \partial B_1, \quad \pa{u}{\nu}=0 \text{ on }\partial W\cap B_1.\]
    
    If the opening angle of $W_2$ is equal to $\pi/2$, then by evenly extending $\bar u$ across $x_1, x_2$ as before, we conclude that $\bar u$ is the weak solution to an elliptic equation with H\"older continuous coefficients in $B_1$, and thus is $C^{2,\alpha}$ in $B_{1/2}(0)$. Thus $u\in C^{2,\alpha}(\overline \Omega\cap B_{1/2}(0))$.
    
    Now suppose that the opening angle of $W_2$ is less than $\pi/2$. Without loss of generality we assume that $\bar u(0)=0$ and $D\bar u(0)=0$, otherwise consider $\bar u - \bar u(0) - D\bar u(0)\cdot x$ instead. For each $k\ge 3$, the function $D_{x_k} \bar u $ is the weak solution to an elliptic equation in the form of \eqref{equation.appendix.euclidean.assumptions}. Thus, by \cite[Theorem 4.2]{Liberman1988holder}, $D_k \bar u\in C^{1,\alpha}(\overline W)$.  Therefore, for any $\phi\in C_0^1(B_{1/4}(0))$, 
    \[\int_{B_{1/4(0)}\cap \overline W} \sum_{i,j=1}^2 \bar a_{ij} D_{x_i} \bar u D_{x_j} \phi = -\int_{B_{1/4}(0)\cap \overline W} \bar F\phi,\]
    where $F$ satisfies that $F\in C^{0,\alpha}(B_{1/4}(0)\cap \overline W)$ and  $F(0)=0$. We not take $\phi = \eta(x_1,x_2) \zeta(x_3,\cdots,x_n)$, where $\eta$ and $\zeta$ are both compactly supported. Thus, the above gives
    \[\int_{B_{1/4}(0)\cap \overline W} \sum_{i,j=1}^2 \bar a_{ij} (D_{x_i} \bar u) (D_{x_j} \eta ) \zeta = \int_{B_{1/4}(0)\cap \overline W} F\eta\zeta.\]
    Since $\bar a_{ij}, D_{x_i} \bar u , F$ are $C^{0,\alpha}$, we can let $\zeta$ approximate the Dirac delta. Thus, letting $v(x_1,x_2) = \bar u(x_1,x_2,0,\cdots,0)$, $a'_{ij}(x_1,x_2) = \bar a_{ij}(x_1,x_2,0,\cdots, 0)$, $F'(x_1,x_2) = F(x_1,x_2,0,\cdots,0)$, we have that $v$ satisfies that
    \[ \int_{B_{1/4}(0)\cap W_2} a'_{ij} D_i v D_j \eta = \int_{B_{1/4(0)}\cap W_2} F'\eta,\quad \forall \eta\in C_0^1(B_{1/4}). \]
    Moreover, $v(0)=Dv(0)=0$. Since the opening angle of $W_2$ is strictly less than $\pi/2$, there is no quadratic terms in the Fourier expansion of a Neumann harmonic function on $W_2$. Using $a'_{ij}(0,0)=\delta_{ij}$ and $a'_{ij}\in C^{0,\alpha}(B_{1/4}(0))$, standard Campanato estimates (see, e.g. \cite[Section 4]{Liberman1988holder}) implies that $v\in C^{2,\alpha}(B_{1/4}(0)\cap W_2)$, with the decay
    \[|v|_{0,B_{r}(0)\cap W_2}\le c(W_2) r^{2+\alpha} |v|_{0,B_{1/4}(0)\cap W_2},\]
    for some $\alpha>0$ that depends only on $W_2$. This shows that $\bar u\in C^{2,\alpha}(B_{1/4}(0)\cap W)$, and thus $u\in C^{2,\alpha}(B_{1/4}(0)\cap \overline \Omega)$.
\end{proof}

We are now ready to conclude the regularity needed for the paper.

\begin{coro}
    Assume $3\le n\le 7$. Let $(M^n,g)$ be an over cubic Riemannian manifold with $g$ a $C^{2,\alpha}$ metric, and the dihedral angle is everywhere $\pi/2$. Let $\Sigma^{n-1}$ be a properly embedded volume minimizing hypersurface of $M$ with free boundary. Then $\Sigma$ is a $C^{2,\alpha}$ graph over its tangent plane everywhere.
\end{coro}

\begin{coro}
    Let $M$ be as above. Suppose $u\in W^{1,2}(M)$ is a weak solution to 
    \[\begin{cases}\Delta_g u + cu=f \quad&\text{in }M,\\ \pa{u}{\nu}=q_ju \quad&\text{on }F_j.\end{cases}\]
    Here $c,f\in C^{0,\alpha}(\overline{M})$, and $q_j\in C^{1,\alpha}(\overline{F_j})$ satisfy the compatibility condition 
    \[\paop{\nu_j}q_i=\paop{\nu_i}q_j\]
    along the intersection $F_i\cap F_j$. Then $u\in C^{2,\alpha}(\overline{M})$.
\end{coro}

\begin{proof}
    By standard elliptic regularity in Lipschitz domains, $u$ is $C^{0,\alpha}$ smooth. Since the functions $q_j$ satisfies the compatibility condition, there exists a $C^{2,\alpha}$ function in the interior of $M$, and on the smooth part of $\partial M$. Also, since $M$ is locally convex, it follows from \cite[Theorem 4.1]{Liberman1988holder} that $u\in C^{1,\alpha}(\overline{M})$.
    
    Since $\{q_j\}$ satisfies the compatibility condition, there exists a function $v_0\in C^{2,\alpha}(\overline{M})$ such that $\pa{v_0}{\nu}=q_j$ on $F_j$. Let $v=e^{v_0}$. Then it is straightforward to check that the function $w=\frac{u}{v}\in W^{2,2}(M)$ is a weak solution to an equation in the form
    \[\begin{cases}\Delta_g w + \vec{b}\cdot \nabla w + cw=f \quad&\text{in }M,\\ \pa{w}{\nu}=0 \quad&\text{on }\partial M.\end{cases}\]
    Since $w\in C^{1,\alpha}(\overline{M})$, by possibly subtracting $lw$ on both sides, where $l$ is a large constant, we may also without loss of generality assume that the constant term $c$ is non-positive. Near each point $p$ at the corner of $M$, consider the equation
    \[\begin{cases}
        \Delta_g w_0+\vec{b}\cdot w_0+cw_0=f \quad &\text{in }M\cap B_\rho(p),\\ \pa{w_0}{\nu}=0 \quad &\text{on }\partial M\cap B_\rho(p),\\ w_0=w \quad &\text{on }M\cap \partial B_\rho(p).
    \end{cases}\]
    By Proposition \ref{proposition.appendix.regularity}, the solution $w$ is in $C^{2,\alpha}(M\cap B_{\rho/2}(p))$. On the other hand, the maximum principle implies that $w=w_0$ in $\overline{M\cap B_\rho(p)}$. We then conclude that $w\in C^{2,\alpha}(\overline{M})$, and hence $u\in C^{2,\alpha}(\overline{M})$.
\end{proof}

\bibliographystyle{amsalpha}

\providecommand{\bysame}{\leavevmode\hbox to3em{\hrulefill}\thinspace}
\providecommand{\MR}{\relax\ifhmode\unskip\space\fi MR }
\providecommand{\MRhref}[2]{%
  \href{http://www.ams.org/mathscinet-getitem?mr=#1}{#2}
}
\providecommand{\href}[2]{#2}
\begin{thebibliography}{}

\end{thebibliography}


\begin{thebibliography}{99}
	
	\bibitem[ABN]{AleksandrovRerestovskiiNikolaev86}
	A.~D. Aleksandrov, V.~N. Berestovski\u\i, and I.~G. Nikolaev, \emph{Generalized
		{R}iemannian spaces}, Uspekhi Mat. Nauk \textbf{41} (1986), no.~3(249), 3--44, 240. MR {854238}, Zbl {0625.53059}
	
	\bibitem[AR]{AndersonRodriguez1989minimal}
	Michael~T. Anderson and Lucio Rodr\'{\i}guez, \emph{Minimal surfaces and
		{$3$}-manifolds of nonnegative {R}icci curvature}, Math. Ann. \textbf{284} (1989), no.~3, 461--475. MR {1001714}, Zbl 0653.53024
	
	\bibitem[B]{bamler2016ricci}
	Richard~H. Bamler, \emph{A {R}icci flow proof of a result by {G}romov on lower
		bounds for scalar curvature}, Math. Res. Lett. \textbf{23} (2016), no.~2, 325--337. MR {3512888}, Zbl 1373.53043
	
	\bibitem[BBN]{BrayBrendleNeves2010rigidity}
	Hubert~L. Bray, Simon Brendle, and Andre Neves, \emph{Rigidity of area-minimizing
		two-spheres in three-manifolds}, Comm. Anal. Geom. \textbf{18} (2010), no.~4,
	821--830. MR {2765731}, Zbl 1226.53054
	
	\bibitem[BKKS]{bray2019harmonic}
	Hubert~L. Bray, Demetre~P. Kazaras, Marcus~A. Khuri, and Daniel~L. Stern,
	\emph{Harmonic functions and the mass of 3-dimensional asymptotically flat
		{R}iemannian manifolds}, arXiv: 1911.06754 (2019).
	
	\bibitem[B1]{Brendle2012rigidity}
	Simon Brendle, \emph{Rigidity phenomena involving scalar curvature}, Surveys in
	differential geometry. {V}ol. {XVII}, Surv. Differ. Geom., vol.~17, Int.
	Press, Boston, MA, 2012, pp.~179--202. MR {3076061}, Zbl 1382.53012
	
	\bibitem[B2]{bray2019scalar}
	Hubert~L. Bray and Daniel~L. Stern, \emph{Scalar curvature and harmonic
		one-forms on three-manifolds with boundary}, Comm. Anal. Geom. (to appear). 	arXiv:1911.06803.
	
	\bibitem[BW]{BoileauWang1996degree}
	Michel Boileau and Shicheng Wang, \emph{Non-zero degree maps and surface
		bundles over {$S^1$}}, J. Differential Geom. \textbf{43} (1996), no.~4,
	789--806. MR {1412685}, Zbl 0868.57029
	
	\bibitem[CCE]{CarlottoChodoshEichmair2016effective}
	Alessandro Carlotto, Otis Chodosh, and Michael Eichmair, \emph{Effective
		versions of the positive mass theorem}, Invent. Math. \textbf{206} (2016),
	no.~3, 975--1016. MR {3573977}, Zbl 1354.53071
	
	\bibitem[CEM]{ChodoshEichmairMoraru2018splitting}
	Otis Chodosh, Michael Eichmair, and Vlad Moraru, \emph{A splitting theorem for
		scalar curvature}, Comm. Pure Appl. Math. \textbf{72} (2019), no.~6,
	1231--1242. MR {3948556},  Zbl 1422.53027
	
	\bibitem[CG]{CaiGallowway2000rigidity}
	Mingliang Cai and Gregory~J. Galloway, \emph{Rigidity of area minimizing tori
		in 3-manifolds of nonnegative scalar curvature}, Comm. Anal. Geom. \textbf{8}
	(2000), no.~3, 565--573. MR {1775139}, Zbl 0994.53028
	
	\bibitem[DPM]{DePhilippisMaggi15regularity}
	Guido De~Philippis and Francesco Maggi, \emph{Regularity of free boundaries in
		anisotropic capillarity problems and the validity of {Y}oung's law}, Arch.
	Ration. Mech. Anal. \textbf{216} (2015), no.~2, 473--568. MR {3317808}, Zbl 1321.35158
	
	\bibitem[E]{Ehrlich1976metricdeformation}
	Paul Ehrlich, \emph{Metric deformations of curvature. {I}. {L}ocal convex
		deformations}, Geometriae Dedicata \textbf{5} (1976), no.~1, 1--23.
	MR {0487886}, Zbl 0345.53024
	
	\bibitem[EL]{EdelenLi2019regularity}
	Nicolous Edelen and Chao Li, \emph{Regularity of free boundary minimal surfaces
		in locally polyhedral domains}, Comm. Pure Appl. Math. (to appear). 	arXiv:2006.15441.
	
	\bibitem[EMW]{EichmairMiaoWang2012extension}
	Michael Eichmair, Pengzi Miao, and Xiaodong Wang, \emph{Extension of a theorem
		of {S}hi and {T}am}, Calc. Var. Partial Differential Equations \textbf{43}
	(2012), no.~1-2, 45--56. MR {2860402}, Zbl 1238.53025
	
	\bibitem[GLZ]{GuangLiZhou2016curvature}
	Qiang Guang, Martin Li, and Xin Zhou, \emph{{Curvature estimates for stable free
			boundary minimal hypersurfaces}}, J. Reine Angew. Math \textbf{2020} (2020),
	no.~759, 245--264. MR {4058180}, Zbl 1466.53075
	
	\bibitem[G1]{Gromov2014Dirac}
	Misha Gromov, \emph{Dirac and {P}lateau billiards in domains with corners},
	Cent. Eur. J. Math. \textbf{12} (2014), no.~8, 1109--1156. MR {3201312}, Zbl 1315.53027
	
	\bibitem[G2]{Gromov2018Adozen}
	\bysame, \emph{A dozen problems, questions and conjectures about positive
		scalar curvature}, Foundations of mathematics and physics one century after
	{H}ilbert, Springer, Cham, 2018, pp.~135--158. MR {3822551}, Zbl 1432.53052
	
	\bibitem[G3]{Gromov2018Metric}
	\bysame, \emph{Metric inequalities with scalar curvature}, Geom. Funct. Anal.
	\textbf{28} (2018), no.~3, 645--726. MR {3816521}, Zbl 1396.53068
	
	\bibitem[K]{Kazaras2019desingularizing}
	Demetre {Kazaras}, \emph{{Desingularizing positive scalar curvature
			4-manifolds}}, arXiv:1905.05306 (2019).
	
	\bibitem[Li]{Li2017polyhedron}
	Chao Li, \emph{{A polyhedron comparison theorem for 3-manifolds with positive
			scalar curvature}}, Invent. Math. \textbf{219} (2019), no.~1, 1--37. MR {4050100}, Zbl 1440.53049
	
	\bibitem[Lib]{Liberman1988holder}
	Gary Liberman, \emph{H\"older continuity of the gradient at a corner for the
		capillary problem and related results}, Pacific J. Math \textbf{133} (1988),
	no.~1, 115--135. MR {0936359}, Zbl 0669.35034
	
	\bibitem[Liu]{liu2013three-manifolds}
	Gang Liu, \emph{3-manifolds with nonnegative {R}icci curvature}, Invent. Math.
	\textbf{193} (2013), no.~2, 367--375. MR {3090181}, Zbl 1279.53035
	
	\bibitem[LM]{LiMantoulidis2018positive}
	Chao Li and Christos Mantoulidis, \emph{Positive scalar curvature with skeleton
		singularities}, Math. Ann. \textbf{374} (2019), no. ~1-2, 99--131. MR {3961306}, Zbl 1418.53042
	
	\bibitem[LY]{LiuYau2006positivity}
	Chiu-Chu~Melissa Liu and Shing-Tung Yau, \emph{Positivity of quasi-local mass.
		{II}}, J. Amer. Math. Soc. \textbf{19} (2006), no.~1, 181--204. MR {2169046}, Zbl 1081.83008
	
	\bibitem[LZ]{LiZhou2017maximum}
	Martin Li and Xin Zhou, \emph{{A maximum principle for free boundary minimal
			varieties of arbitrary codimension}}, Comm. Anal. Geom. (to appear). 	arXiv:1708.05001.
	
	\bibitem[M1]{Miao02}
	Pengzi Miao, \emph{Positive mass theorem on manifolds admitting corners along a
		hypersurface}, Adv. Theor. Math. Phys. \textbf{6} (2002), no.~6, 1163--1182
	(2003). MR {1982695}
	
	\bibitem[M2]{miao2019measuring}
	\bysame, \emph{Measuring mass via coordinate cubes}, Comm. Math. Phys.
	\textbf{379} (2020), no. ~2, 773--783. MR {4156222}, Zbl 1471.58011
	
	\bibitem[MNS]{MaximoNunesSmith2017freeboundary}
	Davi Maximo, Ivaldo Nunes, and Graham Smith, \emph{Free boundary minimal annuli
		in convex three-manifolds}, J. Differential Geom. \textbf{106} (2017), no.~1,
	139--186. MR {3640009}, Zbl 1386.53071
	
	\bibitem[Sc]{Schoen2006minimal}
	Richard Schoen, \emph{Minimal submanifolds in higher codimension}, Mat.
	Contemp. \textbf{30} (2006), 169--199, XIV School on Differential Geometry
	(Portuguese). MR {2373510}, Zbl 1151.53344
	
	\bibitem[Si]{Simon87strict}
	Leon Simon, \emph{A strict maximum principle for area minimizing
		hypersurfaces}, J. Differential Geom. \textbf{26} (1987), no.~2, 327--335.
	MR {906394}, Zbl 0625.53052
	
	\bibitem[ST]{ShiTam2002positivemass}
	Yuguang Shi and Luen-Fai Tam, \emph{Positive mass theorem and the boundary
		behaviors of compact manifolds with nonnegative scalar curvature}, J.
	Differential Geom. \textbf{62} (2002), no.~1, 79--125. MR {1987378}, Zbl 1071.53018
	
	\bibitem[St]{stern2019scalar}
	Daniel Stern, \emph{Scalar curvature and harmonic maps to {$S^1$}}, J.
	Differential. Geom. (to appear). 	arXiv:1908.09754.
	
	\bibitem[SW]{SolomonWhite1989strong}
	Bruce Solomon and Brian White, \emph{A strong maximum principle for varifolds
		that are stationary with respect to even parametric elliptic functionals},
	Indiana Univ. Math. J. \textbf{38} (1989), no.~3, 683--691. MR {1017330}, Zbl 0711.49059
	
	\bibitem[SY1]{SchoenYau1979ProofPositiveMass}
	Richard Schoen and Shing-Tung Yau, \emph{On the proof of the positive mass
		conjecture in general relativity}, Comm. Math. Phys. \textbf{65} (1979),
	no.~1, 45--76. MR {526976}, Zbl 0405.53045
	
	\bibitem[SY2]{SchoenYau1979positivity}
	\bysame, \emph{Positivity of the total mass of a
		general space-time}, Phys. Rev. Lett. \textbf{43} (1979), 1457--1459. MR {0547753}
	
	\bibitem[SY3]{SchoenYau2017positive}
	\bysame, \emph{{Positive scalar curvature and minimal
			hypersurface singularities}}, arXiv: 1704.05490 (2017).
	
	\bibitem[Wh1]{white1987space}
	Brian White, \emph{The space of {$m$}-dimensional surfaces that are stationary
		for a parametric elliptic functional}, Indiana Univ. Math. J. \textbf{36}
	(1987), no.~3, 567--602. MR {905611}, Zbl 0770.58005
	
	\bibitem[Wh2]{white2010maximum}
	\bysame, \emph{The maximum principle for minimal varieties of arbitrary
		codimension}, Comm. Anal. Geom. \textbf{18} (2010), no.~3, 421--432.
	MR {2747434}, Zbl 1226.53061
	
	\bibitem[Wic]{Wickramasekera14strongmaximum}
	Neshan Wickramasekera, \emph{A sharp strong maximum principle and a sharp
		unique continuation theorem for singular minimal hypersurfaces}, Calc. Var.
	Partial Differential Equations \textbf{51} (2014), no.~3-4, 799--812.
	MR {3268871}, Zbl 1320.35109
	
	\bibitem[Wit]{Witten1981newproof}
	Edward Witten, \emph{A new proof of the positive energy theorem}, Comm. Math.
	Phys. \textbf{80} (1981), no.~3, 381--402. MR {626707}, Zbl 1051.83532
	
	\bibitem[WY]{WangYau2009isometric}
	Mu-Tao Wang and Shing-Tung Yau, \emph{Isometric embeddings into the {M}inkowski
		space and new quasi-local mass}, Comm. Math. Phys. \textbf{288} (2009),
	no.~3, 919--942. MR {2504860}, Zbl 1195.53039
	
\end{thebibliography}

\end{document}